\documentclass[reqno,11pt]{amsart}
\usepackage{multind}
\usepackage{graphicx}
\usepackage{amscd}
\usepackage[latin1]{inputenc} 
\usepackage{pst-all}          
\usepackage[mathscr]{eucal}
\numberwithin{equation}{section}
\allowdisplaybreaks[1]

\newcommand{\SetFigFont}[3]{}

\title[Some curvature problems in semi-Riemannian geometry]{Some curvature problems in semi-Riemannian geometry}

\author[F.\ Finster]{Felix Finster}
\thanks{Supported by the Deutsche Forschungsgemeinschaft within the
Priority Program ``Globale Differentialgeometrie.''}
\address{NWF I - Mathematik \\ Universit\"at Regensburg \\ D-93040 Regensburg \\ Germany}
\email{Felix.Finster@mathematik.uni-r.de \\
Marc.Nardmann@mathematik.uni-r.de}

\author[M.\ Nardmann]{Marc Nardmann}

\theoremstyle{definition} 
\newtheorem{Def}{Definition}[section]

\theoremstyle{plain}      
\newtheorem{Thm}[Def]{Theorem}
\newtheorem{Prp}[Def]{Proposition}

\newcommand{\Thanks}{\vspace*{.5em} \noindent \thanks}
\newcommand{\beq}{\begin{equation}}
\newcommand{\eeq}{\end{equation}}
\newcommand{\Proof}{\begin{proof}}
\newcommand{\QED}{\end{proof} \noindent}

\newcommand{\R}{\mathbb{R}}

\newcommand{\N}{\mathbb{N}}

\newcommand{\D}{\mathcal{D}}

\renewcommand{\O}{{\mathscr{O}}}

\newcommand{\spec}{\mbox{\rm spec}}
\makeindex{sectornotation}
\makeindex{sectorsubject}

\newcommand{\OO}{\text{\rm O}}
\newcommand{\Gr}{\text{\rm Gr}}
\newcommand{\set}[1]{\{#1\}}
\newcommand{\without}{\setminus}
\newcommand{\abs}[1]{\lvert#1\rvert}
\newcommand{\Tw}{\textsl{Tw}}
\newcommand{\Sw}{\textsl{Sw}}
\newcommand{\scal}{s}
\newcommand{\Ric}{\text{\rm Ric}}
\newcommand{\suchthat}{\mathrel{|}}

\newcommand{\Hess}{\text{\rm Hess}}
\newcommand{\divergence}{\text{\rm div}}
\newcommand{\laplace}{\Delta}
\newcommand{\Twist}[2]{\textsl{Tw}^{#1}_{#2}}
\newcommand{\Symst}[2]{\textsl{Sw}^{#1}_{#2}}

\newcommand{\mfbd}{\partial}
\newcommand{\pr}{\text{\rm pr}}
\newcommand{\compose}{\circ}
\newcommand{\M}{\mathcal{M}}

\definecolor{green}{rgb}{0.,0.,0.}
\definecolor{red}{rgb}{0.,0.,0.}
\definecolor{blue}{rgb}{0.,0.,0.}
\definecolor{grey}{rgb}{0.7,0.7,0.7}

\begin{document}

\begin{abstract}
In this survey article we review several results on the curvature of semi-Riemannian metrics which are motivated by the positive mass theorem. The main themes are estimates of the Riemann tensor of an asymptotically flat manifold and the construction of Lorentzian metrics which satisfy the dominant energy condition.
\end{abstract}

\maketitle

\tableofcontents


In this survey article we review recent progress on several curvature problems in
semi-Riemannian geometry, each of which has a certain relation to the positive mass theorem (PMT). The focus is on work in which we were involved within the Priority Program ``Globale Differentialgeometrie.''

\smallskip
The time-symmetric version of the PMT says in particular that an asymptotically flat Riemannian manifold with zero mass is flat. Section \ref{SecOne} investigates whether an asymptotically flat manifold whose mass is \emph{almost zero} must be almost flat in a suitable sense. The general, not necessarily time-symmetric, situation is considered as well. The main tool in this work is the spinor which occurs in Witten's proof of the PMT.

\smallskip
The PMT implies that if the energy $E$ and the momentum $P$ of an asymptotically flat spacelike hypersurface $M$ of a Lorentzian manifold in which the dominant energy condition holds satisfy $E=\abs{P}$, then the Lorentzian metric is flat along $M$. Schoen--Yau proved that in this situation, $M$ with its given second fundamental form can be isometrically embedded as a spacelike graph into Minkowski space-time. The short Section \ref{SecTwo} presents an alternative proof of this fact, based on the Lorentzian version of the fundamental theorem of hypersurface theory due to Bär--Gauduchon--Moroianu.

\smallskip
Section \ref{SecThree} deals with the question which smooth manifolds admit a Lorentzian metric that satisfies the dominant energy condition. Since every closed or asymptotically flat spacelike hypersurface of a Lorentzian manifold can potentially yield a PMT-like obstruction to the dominant energy condition, one should avoid in the construction of dominant energy metrics that such spacelike hypersurfaces exist at all. This can indeed be accomplished in many situations.


\section{Analysis of Asymptotically Flat Manifolds via Witten Spinors} \label{SecOne}
Asymptotically flat Lorentzian manifolds describe isolated gravitating systems (like a star
or a galaxy) in the framework of general relativity. As discovered by Arnowitt, Deser and
Misner~\cite{adm}, to an asymptotically flat Lorentzian manifold one can associate
the total energy and the total momentum, defined globally via the asymptotic behavior of the metric
near infinity. Moreover, the energy-momentum tensor gives a local concept of energy and momentum.
These global and local quantities are linked by Einstein's field equations, giving rise to an
interesting interplay between local curvature and the global geometry of space-time.
The first result which shed some light on the nature of this interplay is the positive energy
theorem~\cite{schoen+yau, schoen+yau2}, which states that if the local energy density is positive
(in the sense that the dominant energy condition holds), then the total energy is also positive. More
recently, the proof of the Riemannian Penrose inequalities \cite{huisken+ilmanen, bray} showed that
in the time-symmetric situation, the total energy is not only positive, but it is even larger than
the energy of the black holes, as measured by the surface area of their horizons.
Despite this remarkable progress, many important problems remain open (see for
example~\cite{mars}).

The aim of our research project was to get a better understanding of how total energy and momentum
control the geometry of space-time. In the special case that energy and momentum vanish, the positive
energy theorem yields that the space-time manifold is flat~\cite{schoen+yau2, parker+taubes}.
This suggests that if total energy and momentum are small,
then the manifold should be almost flat, meaning that curvature should be small. But is this conjecture
really correct? Suppose we consider a sequence of space-time metrics such that
total energy and momentum tend to zero. In which sense do the metrics converge to the flat Minkowski
metric?

Although our considerations could not give definitive answers to these questions, at least they
led to a few inequalities giving some geometric insight, as we will outline in what
follows. For simplicity, we begin
in the Riemannian setting (in general dimension~$n$), whereas the generalizations to
include the second fundamental form will be explained in Section~\ref{secsff}.
All our methods use the Witten spinor as introduced in~\cite{witten}. But in contrast to the spinor
proof of the positive energy theorem~\cite{parker+taubes}, we consider second derivatives
of the Witten spinor~$\psi$. Our starting point is a basic inequality involving the $L^2$-norm of the
second derivatives of~$\psi$ (Section~\ref{secL2}).
Using Sobolev techniques, we deduce curvature estimates, which
however involve the isoperimetric constant of the manifold (Section~\ref{secce1}).
An analysis of the level sets of~$|\psi|$ allows us to get estimates which are
independent of the isoperimetric constant but instead involve a volume bound
(Section~\ref{seclevel}). In the case of an asymptotically Schwarzschild space-time, we
then derive weighted $L^2$-estimates of~$\psi$ which involve the lowest eigenvalue~$\lambda$ of
the Dirac operator on a conformal compactification (Section~\ref{secweighted}).
These weighted~$L^2$-estimates finally give rise to
curvature estimates which involve the global geometry of the manifold
only via~$\lambda$ (Section~\ref{secce2}).
We conclude by an outlook and a discussion of open problems (Section~\ref{secoutlook}).

\subsection{The Riemannian Setting, Asymptotically Flat Manifolds} \label{secriem}
We briefly recall how the Riemannian setting arises within the framework of
general relativity. Suppose that space-time is described by the Lorentzian manifold~$(N^4, \bar{g})$,
which for simplicity we will assume to be orientable and time-orientable.
To describe the splitting into space and time as experienced by an observer, one chooses a
foliation of~$N^4$ by spacelike hypersurfaces. Considering the situation at a fixed observer time,
one restricts attention to one hypersurface~$M^3$ of this foliation. Then~$\bar{g}$ induces on~$M^3$
a Riemannian metric~$g_{ij}$. Furthermore, choosing on~$M$ a future-directed normal unit vector
field~$\nu$, we obtain on~$M$ the second fundamental form~$h_{ij} = (\bar{\nabla}_j \nu)_k$.
The {\em{time-symmetric}} situation is obtained by assuming that~$h$ vanishes identically.
This condition is in particular satisfied if the unit normal~$\nu$ is a Killing field, meaning that the system
is static. In this special case, the geometry at the fixed observer time is completely described
by the Riemannian metric~$g$ on~$M^3$.

The physical condition that the local energy density should be positive gives rise to the
dominant energy condition (see~\cite[Section~4.3]{hawking+ellis} and Section \ref{secdomenergy} in the present article) for the energy-momentum tensor.
Using the Einstein equations, it can also be expressed in terms of the Ricci tensor on~$N^4$.
In the time-symmetric situation, the dominant energy condition reduces to the condition that~$(M^3, g)$
should have {\em{non-negative scalar curvature}}. Every orientable, three-dimensional
Riemannian manifold is spin (see for example~\cite{LaMi}). Therefore, it is a sensible mathematical generalization to consider in what follows a {\em{spin}} manifold~$(M^n, g)$ of dimension~$n \geq 3$ of
non-negative scalar curvature. Moreover, in order to exclude singularities, we shall assume
that~$(M^n, g)$ is {\em{complete}}.

Having isolated gravitating systems in mind, we next want to impose that the
Riemannian metric should approach the Euclidean metric in the ``asymptotic ends'' describing
space near infinity. More precisely, considering for simplicity one asymptotic end, the
manifold~$(M^n,g)$ is said to be {\em{asymptotically flat}} if there is a compact
set~$K \subset M$ and a diffeomorphism $\Phi \,:\, M \setminus K \rightarrow
\R^n \setminus B_{\rho}(0)$, $\rho>0$, such that
\beq \label{gdecay}
(\Phi_* g)_{ij} \;=\; \delta_{ij} \:+\: \O(r^{2-n}) \:,\quad
\partial_k (\Phi_* g)_{ij} \;=\; \O(r^{1-n}) \:,\quad
\partial_{kl} (\Phi_* g)_{ij} \;=\; \O(r^{-n}) \:.
\eeq
These decay conditions imply that scalar curvature is of the order~$\O(r^{-n})$.
We need to make the stronger assumption that {\em{scalar curvature is integrable}}.
In the Riemannian setting, the total energy is also referred to as the {\em{total mass}}~$m$
of the manifold (whereas total momentum vanishes). It is defined by
\begin{equation}
m \;=\; \frac{1}{c(n)} \lim_{\rho \rightarrow \infty}
\int_{S_\rho} (\partial_j (\Phi_* g)_{ij} - \partial_i (\Phi_* g)_{jj})
\:d\Omega^i \:, \label{massdef}
\end{equation}
where $c(n)>0$ is a normalization constant
and $d\Omega^i$ denotes the product of the volume form
on~$S_\rho \subset \R^n$ by the $i^\text{th}$ component of the normal vector on $S_\rho$ (also we use the Einstein summation convention and sum over all indices which appear twice).
The definition~\eqref{massdef} was first given in~\cite{adm}. In~\cite{bartnik} it is proved
that the definition is independent of the choice of~$\Phi$.
The {\em{positive mass theorem}}~\cite{schoen+yau} states that~$m \geq 0$ in the case~$n \leq 7$
(working even for non-spin manifolds). An alternative proof using spinors is given
by~\cite{witten, parker+taubes} and in general dimension in~\cite{bartnik}.

\subsection{An $L^2$-Estimate for the Second Derivatives of the Witten Spinor} \label{secL2}
Before introducing our methods, we briefly recall the spinor proof of the positive mass theorem.
The basic reason why spinors are very useful for the analysis of asymptotically flat
spin manifolds is the {\em{Lichnerowicz-Weitzenb{\"o}ck formula}}
\beq \label{LW}
\D^2 \;=\; -\nabla^2 + \frac{s}{4} \;,
\eeq
which actually goes back to Schr\"odinger~\cite{schroedinger}. Here~$\D$ is the Dirac operator,
$\nabla$ is the spin connection, and~$s$ denotes scalar curvature.
Witten~\cite{witten} considered solutions of the Dirac equation with constant
boundary values~$\psi_0$ in the asymptotic end,
\beq \label{bvp}
\D \psi \;=\; 0 \:,\qquad
\lim_{|x| \rightarrow \infty} \psi(x) \;=\; \psi_0 \quad \text{with} \quad
|\psi_0|=1\:,
\eeq
where~$\psi$ is a smooth section of the spinor bundle~$SM$.
In~\cite{parker+taubes, bartnik} it is proved that for any~$\psi_0$,
this boundary value problem has a unique solution. We refer to~$\psi$ as the
{\em{Witten spinor}} with boundary values~$\psi_0$.
For a Witten spinor, the Lichnerowicz-Weitzenb\"ock formula implies that
\beq \label{div}
\nabla_i \,\langle \psi, \nabla^i \psi \rangle \;=\;
|\nabla \psi|^2 + \frac{s}{4}\: |\psi|^2\:.
\eeq
Integrating over~$M$, applying Gauss' theorem and relating the boundary
values at infinity to the total mass
(where we choose~$c(n)$ in~(\ref{massdef}) appropriately),
one obtains the identity~\cite{witten, parker+taubes, bartnik}
\beq \label{ibp}
\int_M \left( |\nabla \psi|^2 + \frac{s}{4}\, |\psi|^2 \right) d\mu_M \;=\; m\:.
\eeq
As the integrand is obviously non-negative, this identity immediately implies the positive mass theorem for spin manifolds.

We now outline the derivation of an $L^2$-estimate of~$\nabla^2 \psi$ (for details
see~\cite{mass} and~\cite{curv}). We consider similar to~\eqref{div} a divergence,
but now of an expression involving higher derivatives,
\[ \nabla_i \,\langle \nabla_j \psi, \nabla^i \nabla^j \psi \rangle \;=\;
|\nabla^2 \psi|^2 + \langle \nabla_j \psi, \nabla_i \nabla^i \nabla^j \psi \rangle \:. \]
In the third derivative term, we commute~$\nabla^j$ to the left,
\[ \nabla_i \nabla^i \nabla^j \psi = \left[ \nabla_i \nabla^i, \nabla^j \right] \psi
+ \nabla^j \left( \nabla_i \nabla^i \psi \right) . \]
Then in the last summand we can again apply the Lichnerowicz-Weitzenb\"ock formula,
whereas the commutator gives rise to curvature terms. We integrate the resulting
equation over~$M$. Using the faster decay of the higher derivatives of~$\psi$,
integrating by parts does not give boundary terms. Using the the H\"older inequality together
with the inequality
\beq \label{npsi}
\int_M |\nabla \psi|^2 d\mu_M \leq m
\eeq
(which is obvious from~\eqref{ibp}), we obtain the estimate
\beq \boxed{ \quad
\int_M |\nabla^2\psi|^2\,d\mu_M \leq
m \:C_1(n)\: \sup_M |R|
+ \sqrt{m}\:C_2(n)\: \| \nabla R\|_{L^2(M)}\, \sup_M |\psi| \, ,
\quad } \label{L2est}
\eeq
where~$|R| = \sqrt{R_{ijkl} R^{ijkl}}$ denotes the norm of the Riemann tensor.
We remark that in~\cite{mass, curv, kraus} a more general inequality for
$\int_M \eta |\nabla^2\psi|^2\,d\mu_M$ with an arbitrary smooth function~$\eta$ is
considered. By choosing~$\eta$ to be a test function, this makes it possible to ``localize''
the inequality to obtain curvature estimates on the support of~$\eta$.
For simplicity, in this survey article the function~$\eta$ will always be omitted.

\subsection{Curvature Estimates Involving the Isoperimetric Constant} \label{secce1}
In short, curvature estimates are obtained from~\eqref{L2est} by estimating the spinors
by suitable a-priori bounds. We first outline how to treat the second derivative term~$|\nabla^2 \psi|^2$
(for details see~\cite{mass} and~\cite{curv}). The Schwarz inequality yields
\[ \big\langle [\nabla_i, \nabla_j] \psi,  [\nabla_i, \nabla_j] \psi \big\rangle \leq 4\, |\nabla^2 \psi|^2 \:. \]
Rewriting the commutators by curvature, we obtain an expression which
is quadratic in the Riemann tensor. In dimension~$n=3$, one can use the properties of the Clifford
multiplication to obtain
\[ |R|^2\, |\psi|^2 \leq c(n)\: |\nabla^2 \psi|^2\:. \]
In dimension~$n>3$, this inequality is in general wrong. But we get a similar inequality
for a family~$\psi_1, \ldots, \psi_N$ of Witten spinors,
\beq \label{es1}
\sum_{i=1}^N |R|^2\, |\psi_i|^2 \leq c(n)\: \sum_{i=1}^N |\nabla^2 \psi_i|^2\:,
\eeq
where the boundary values~$\lim_{|x| \rightarrow \infty} \psi_i(x)$ form an orthonormal
basis of the spinors at infinity. The family of Witten spinors can be handled most conveniently
by forming the so-called {\em{spinor operator}} (for details see~\cite{curv}).

We next consider the term~$\sup_M |\psi|$ in~\eqref{L2est}. A short calculation using the
Lichnero\-wicz-Weitzenb\"ock formula shows that~$|\psi|$ is {\em{subharmonic}},
(see~\cite[Section~2]{level}),
\beq \label{subharmonic}
\Delta |\psi| \geq \frac{s}{4}\: |\psi| \geq 0 \:.
\eeq
Thus the maximum principle yields that~$|\psi|$ has no interior maximum, and in view
of the boundary conditions at infinity~\eqref{bvp} we conclude that
\beq \label{es2}
\sup_M |\psi| =1 \:.
\eeq

Using~\eqref{es1} and~\eqref{es2} in~\eqref{L2est}, we obtain the estimate
\beq \label{nes}
\int_M |R|^2 \,\Big( \sum\nolimits_{i=1}^N |\psi_i|^2 \Big) \,d\mu_M \leq
m \:C_1(n)\: \sup_M |R| + \sqrt{m}\:C_2(n)\: \|\nabla R\|_{L^2(M)}\, .
\eeq
The remaining task is to estimate the norm of the spinors {\em{from below}}.
Such estimates are difficult to obtain, partly because the norm of the spinor depends
sensitively on the unknown geometry of~$M$. We now begin with the simplest estimates,
whereas more refined methods will be explained in Sections \ref{seclevel} and \ref{secce2}.

The inequality~\eqref{npsi} tells us that, for small~$m$, the derivative of the spinor is
small in the $L^2$-sense, suggesting that in this case the spinor should be almost constant,
implying that~$|\psi|$ should be bounded from below. In order to make this argument precise,
we set~$f=1-|\psi|$ and use the Kato inequality~$|\nabla f| \leq |\nabla \psi|$ to
obtain~$\| \nabla f \|_{L^2(M)} \leq m$. The Sobolev inequality (see~\cite[Section~4]{curv})
for details) implies that
\[ \| f \|_{L^q(M)} \leq \frac{q}{k}\:m \qquad \text{where} \qquad
q= \frac{2n}{n-2} \:, \]
and~$k$ denotes the {\em{isoperimetric constant}}. Thus we only get an integral estimate of~$f$.
But this integral bound also implies that~$f$ is pointwise small, except on a set of small measure.
We thus obtain the following result (see~\cite[Theorem~1.2]{curv}).

\begin{Thm} \label{thm1}
Let $(M^n, g)$, $n \geq 3$, be a complete asymptotically flat Riemannian spin manifold of
non-negative scalar curvature. Then there is a set $\Omega \subset M$ with
\begin{equation}
    \mu(\Omega) \leq \left( \frac{c_3\:m}{k^2} \right)^{\frac{n}{n-2}}
    \label{eq:ta}
\end{equation}
such that the following inequality holds,
\beq \label{cint1}
\int_{M \setminus \Omega} |R|^2 \:d\mu_M \leq
m \:c_1(n)\: \sup_M |R| \:+\:
\sqrt{m}\:c_2(n)\: \| \nabla R\|_{L^2(M)} \;.
\eeq
\end{Thm}
This theorem quantifies that  the manifold indeed becomes flat in the limit~$m \searrow 0$,
provided that~$\sup_M |R|$ and~$\| \nabla R\|_{L^2(M)}$
are uniformly bounded and that the isoperimetric constant is bounded away from zero.
The appearance of the isoperimetric constant and of the exceptional set~$\Omega$
can be understood from the following simple example. We choose on~$M^3=\R^3$
the Schwarzschild metric~$g_{ij}(x) = (1-2m/ |x|)^4 \,\delta_{ij}$ (in order to clarify the
connection to the construction in Section~\ref{secriem}, we remark that this~$g_{ij}$
is isometric to the induced Riemannian metric on the $t=\text{const}$ slice of the standard
Schwarzschild space-time). For the geometric understanding, it is helpful to
isometrically embed~$M^3$ into the Euclidean~$\R^4$ (see the left of Figure~\ref{fig2}).
This shows that~$M^3$ has two asymptotic ends, one as~$|x| \rightarrow \infty$ and
the other as~$|x| \rightarrow 0$. The minimal hypersurface~$r=m/2$ has the interpretation
as the {\em{event horizon}}.
\begin{figure}[t]
\begin{picture}(0,0)%
\includegraphics{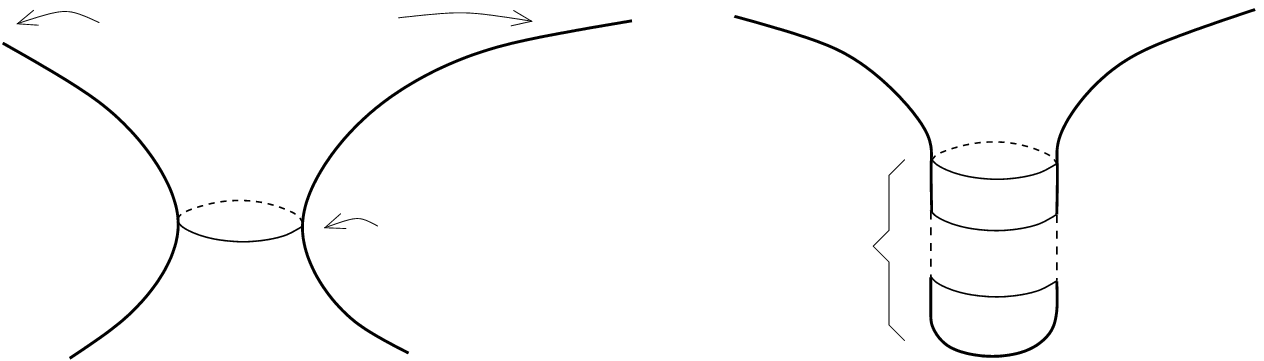}%
\end{picture}%
\setlength{\unitlength}{1657sp}%
\begingroup\makeatletter\ifx\SetFigFontNFSS\undefined%
\gdef\SetFigFontNFSS#1#2#3#4#5{%
  \reset@font\fontsize{#1}{#2pt}%
  \fontfamily{#3}\fontseries{#4}\fontshape{#5}%
  \selectfont}%
\fi\endgroup%
\begin{picture}(14376,4176)(1723,-1444)
\put(3023,2369){\makebox(0,0)[lb]{\smash{{\SetFigFontNFSS{11}{14.4}{\rmdefault}{\mddefault}{\updefault}asymptotic end}}}}
\put(11112,-289){\makebox(0,0)[lb]{\smash{{\SetFigFontNFSS{11}{14.4}{\rmdefault}{\mddefault}{\updefault}$L$}}}}
\put(6173,-46){\makebox(0,0)[lb]{\smash{{\SetFigFontNFSS{11}{14.4}{\rmdefault}{\mddefault}{\updefault}horizon $\displaystyle |x|=\frac{m}{2}$}}}}
\end{picture}%
\caption{The Schwarzschild metric (left) and the manifold after gluing (right).}
\label{fig2}
\end{figure}
In order to get a manifold with one asymptotic end, we cut~$M^3$ at the event horizon
and glue in a cylinder of length~$L$ as well as a spherical cap (see the right of Figure~\ref{fig2}).
This manifold clearly has non-negative scalar curvature.
In the limit~$m \searrow 0$, the resulting manifold becomes flat outside the event horizon.
The region inside the event horizon, however, does not become flat, because the radius of the cylinder
shrinks to zero. This explains why we need an exceptional set. In the limit~$L \rightarrow \infty$,
the volume of this exceptional set necessarily tends to infinity. This is in agreement with~\eqref{eq:ta}
because in this limit, the isoperimetric constant tends to zero.

\subsection{A Level Set Analysis, Curvature Estimates Involving a Volume Bound} \label{seclevel}
The last example explains why working with the volume of an exceptional set might not
be the best method. Namely, in the situation of Figure~\ref{fig2}, it seems preferable
to consider the {\em{surface area}} of the exceptional set. Then cutting at the event horizon,
the long cylinder has disappeared, and we no longer need to worry about the limit when~$L$
gets large. Working with the surface area also seems preferable for physical reasons.
First, as the interior of a black hole is not accessible to measurements, our estimates should
not depend on the geometry inside the event horizon. Therefore, choosing the exceptional set~$\Omega$
such that it contains the interior of the event horizon, our estimates should not depend on
the volume of~$\Omega$, only on its surface area.
Second, the Riemannian Penrose inequalities yield that if the total mass is small, the area of the event horizon is also small. Thus we can hope that there should be an exceptional set of small surface area.

The basic question is how to choose the exceptional set~$\Omega$. In view of the estimate \eqref{nes},
it is tempting to choose the exceptional set as the set where the Witten spinor
(or similarly the spinor operator) is small, i.e.
\beq \label{Odef}
\Omega(\tau) = \{x \in M \text{ with } |\psi(x)| < \tau \}
\eeq
for some~$\tau \in (0,1]$.
This has the advantage that in the region $M \setminus \Omega$, the Witten spinor
is by construction bounded from below by~$\tau$, so that~\eqref{es2} immediately gives rise
to a curvature estimate. Clearly, the resulting estimates are of use only if the exceptional set
is small, for example in the sense that it has small surface area. This consideration was our
motivation for analyzing the {\em{level sets}}
of the Witten spinor~\cite{level}. We here outline a few results of this analysis.

We set~$\phi = |\psi|$ and introduce the functional
\beq \label{Fdef}
F(\tau) = \int_{\Omega(\tau)} |D \phi|^2\, d\mu_M\:.
\eeq
Using the Lichnerowicz-Weitzenb\"ock formula, it is straightforward to verify that this
functional is convex. Moreover, combining the co-area formula and the Schwarz inequality,
one finds that for all~$t_0, t_1$ with~$0 < t_0 < t_1 < 1$, the area~$A$ and the volume~$V$
of the sets $\Omega(\tau)$ and~$\Omega(\tau')$ are related by
\[ \int_{t_0}^{t_1} A(\sigma)\, d\sigma \;\leq\;
\sqrt{ \left( V(t_1) - V(t_0) \right) \left(F(t_1) - F(t_1) \right) } \:. \]
Using the mean value theorem, there is~$t \in [t_0, t_1]$ with
\[ A(t) \;\leq\; \sqrt{F(t_1)-F(t_0)}\; \frac{\sqrt{V(t_1) - V(t_0)}}{t_1-t_0} \:. \]
Furthermore, Sard's lemma can be used to arrange that~$A(t)$ is a hypersurface.
Choosing the exceptional set~$\Omega = \Omega(t)$, the inequality~\eqref{nes} gives
rise to the following curvature estimate.

\begin{Thm} \label{thmg1}
Let~$(M^n, g)$, $n \geq 4$,
be a complete, asymptotically flat manifold whose
scalar curvature is non-negative and integrable.
Suppose that for an interval~$[t_0, t_1]
\subset (0,1]$ there is a constant~$C$ such that every
Witten spinor~(\ref{bvp}) satisfies the volume bound
\beq \label{volume}
V(t_1) - V(t_0) \;\leq\; C \:.
\eeq
Then there is an open set~$\Omega \subset M$ with the following
properties. The $(n-1)$-dimensional Hausdorff measure~$\mu_{n-1}$ of the
boundary of~$\Omega$ is bounded by
\[ \mu_{n-1}(\partial \Omega) \;\leq\; \sqrt{m}\;c_0(n,t_0)\; \frac{\sqrt{C}}{t_1-t_0} \:. \]
On the set~$M \setminus \Omega$, the Riemann tensor satisfies the inequality
\[ \int_{M \setminus \Omega} |R|^2 \;\leq\;
m \:c_1(n,t_0)\: \sup_M |R| \:+\:
\sqrt{m}\:c_2(n,t_0)\: \|\nabla R\|_{L^2(M)} \:. \]
\end{Thm}
For clarity, we point out that~\eqref{volume} only involves the volume of the
region
\beq \label{omegam}
\Omega(t_1) \setminus \Omega(t_0) = \{ x \text{ with } t_0 \leq |\psi| < t_1 \}\:.
\eeq
Thus in order to apply our theorem to the example of Figure~\ref{fig2}, we can choose~$t_0$
such that~$\Omega(t_0)$ just includes the region inside the event horizon. Then the
statement of the theorem no longer depends on the parameter~$L$.
This consideration also explains how it is possible
that Theorem~\ref{thmg1} no longer involves the isoperimetric constant.

\subsection{Weighted $L^2$-Estimates of the Witten Spinor} \label{secweighted}
The curvature estimate in Theorem~\ref{thmg1} has the disadvantage that it involves
the a-priori bound~\eqref{volume} on the volume of the region~$\Omega(t_1) \setminus
\Omega(t_0)$. Since in this region, the spinors are bounded from above and below,
the volume bound could be obtained from an $L^p$-estimate of the Witten spinor for
any~$p < \infty$. Our search for such estimates led to the weighted $L^2$-estimates
in~\cite{weighted}, which we now outline.  A point of general interest is that these estimates involve
the smallest eigenvalue of the Dirac operator on a conformal compactification of~$M$,
thus giving a connection to spectral geometry.

For technical simplicity, the weighted $L^2$-estimates were derived under the additional assumption
that space-time is {\em{asymptotically Schwarzschild}}. Thus we assume that there is a
a compact set~$K \subset M$ and a diffeomorphism $\Phi \,:\, M \setminus K \rightarrow
\R^n \setminus B_{\rho}(0)$, $\rho>0$, such that
\[ (\Phi_* g)_{ij} \;=\; \left(1+\frac{1}{|x|^{n-2}}\right)^{\frac{4}{n-2}}\:\delta_{ij}\:. \]
Then outside the compact set, the metric is conformally flat, and thus by a conformal transformation
\begin{equation} \label{cc}
\tilde{g} \;=\; \lambda^2\: g
\end{equation}
with a smooth function~$\lambda$ with~$\lambda|_K \equiv 1$ we can arrange
that~$\tilde{g}|_{M \setminus K}$ is isometric to a spherical cap of radius~$\sigma$
with the north pole removed.
By adding the north pole, we obtain the complete manifold~$(\bar{M}, \tilde{g})$, being a
conformal one-point compactification of~$(M, g)$ (see Figure~\ref{fig1}).
\begin{figure}[t]%
\begin{picture}(0,0)%
\includegraphics{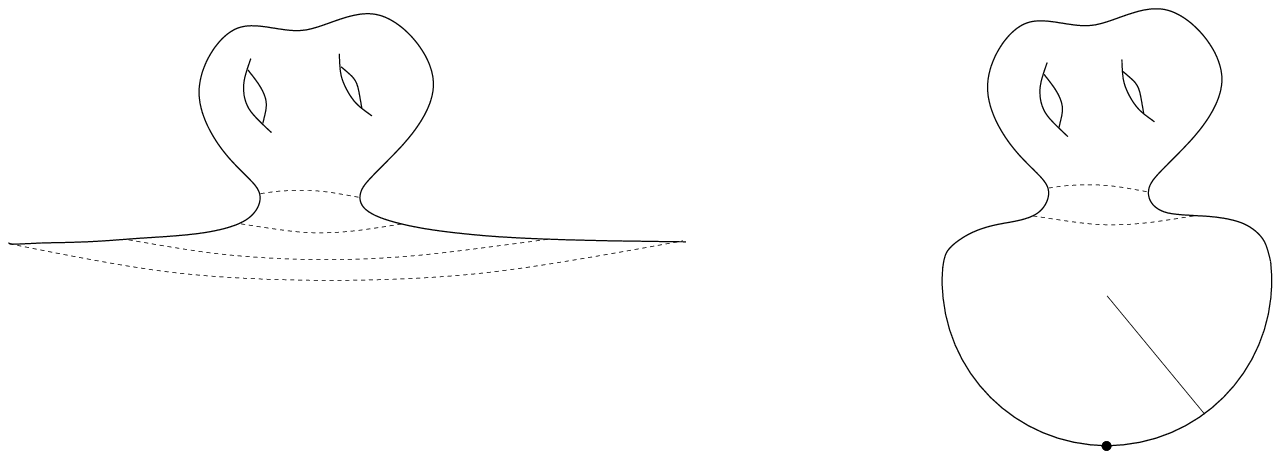}%
\end{picture}%
\setlength{\unitlength}{1243sp}%
\begingroup\makeatletter\ifx\SetFigFont\undefined
\def\x#1#2#3#4#5#6#7\relax{\def\x{#1#2#3#4#5#6}}%
\expandafter\x\fmtname xxxxxx\relax \def\y{splain}%
\ifx\x\y   
\gdef\SetFigFont#1#2#3{%
  \ifnum #1<17\tiny\else \ifnum #1<20\small\else
  \ifnum #1<24\normalsize\else \ifnum #1<29\large\else
  \ifnum #1<34\Large\else \ifnum #1<41\LARGE\else
     \huge\fi\fi\fi\fi\fi\fi
  \csname #3\endcsname}%
\else
\gdef\SetFigFont#1#2#3{\begingroup
  \count@#1\relax \ifnum 25<\count@\count@25\fi
  \def\x{\endgroup\@setsize\SetFigFont{#2pt}}%
  \expandafter\x
    \csname \romannumeral\the\count@ pt\expandafter\endcsname
    \csname @\romannumeral\the\count@ pt\endcsname
  \csname #3\endcsname}%
\fi
\fi\endgroup
\begin{picture}(19443,7088)(-6944,-8934)
\put(-1154,-3436){\makebox(0,0)[lb]{\smash{\SetFigFont{11}{13.2}{rm}$K$}}}
\put(10801,-3391){\makebox(0,0)[lb]{\smash{\SetFigFont{11}{13.2}{rm}$K$}}}
\put(9841,-8656){\makebox(0,0)[lb]{\smash{\SetFigFont{11}{13.2}{rm}$\mathfrak{n}$}}}
\put(8101,-7636){\makebox(0,0)[lb]{\smash{\SetFigFont{11}{13.2}{rm}$C$}}}
\put(10786,-7081){\makebox(0,0)[lb]{\smash{\SetFigFont{11}{13.2}{rm}$\sigma$}}}
\put(5971,-2281){\makebox(0,0)[lb]{\smash{\SetFigFont{11}{13.2}{rm}$(\bar{M}, \tilde{g})$}}}
\put(-6944,-2266){\makebox(0,0)[lb]{\smash{\SetFigFont{11}{13.2}{rm}$(M, g)$}}}
\put(1966,-5401){\makebox(0,0)[lb]{\smash{\SetFigFont{11}{13.2}{rm}$M \setminus K$}}}
\end{picture}%
\caption{The asymptotically Schwarzschild manifold~$(M,g)$ and its
conformal compactification~$(\bar{M}, \tilde{g})$.}%
\label{fig1}%
\end{figure}
The manifold~$(\bar{M}, \tilde{g})$ is again spin. We denote its Dirac operator by~$\tilde{\D}$.

In order to improve the decay properties of the spinor
at infinity, in the asymptotic end we subtract from~$\psi$
a constant spinor multiplied by a function coming from the conformal weight of the sphere,
\[ \delta \psi(x) = \psi(x) - \left( 1 + \frac{1}{|x|^{n-2}}
\right)^{-\frac{n-1}{n-2}}\: \psi_0 \qquad \text{on $M \setminus K$}\:. \]
Under these assumptions, in~\cite{weighted} we prove the following theorem.
\begin{Thm} \label{thmweighted}
Every Witten spinor satisfies the inequality
\[ \int_K \|\psi(x)\|^2 \:dx \:+\:
\int_{M \setminus K} \|\delta \psi(x)\|^2\:\lambda(x)\:dx
\;\leq\; c(n)\: \frac{(\rho+1)^n}{\sigma^2\, \inf \spec({\tilde{\D}}^2)} \:. \]
\end{Thm}

We now sketch the main steps of the proof, also explaining how the
infimum of the spectrum of the operator~${\tilde{\D}}^2$ enters.
Our first step is to get a connection between the
conformally transformed spinor operator and
a quadratic expression in the Dirac Green's function on~$\bar{M}$.
After subtracting suitable counter terms, we can integrate this expression
over~$\bar{M}$ to obtain the Green's function~$G$
of the square of the Dirac operator minus suitable counter
terms. Then our task becomes to analyze the behavior of~$G$
near the pole~$\mathfrak{n}$ of the spherical cap. This is
accomplished by taking the difference of~$G$ and the Green's function
on the sphere and using Sobolev techniques inside the spherical cap.
In this analysis, we need to estimate the sup-norm of~$G$ in the Hilbert
space~$L^2(\bar{M}, S\bar{M})$ by
\[ \| G \| = \sup \spec (G) = \frac{1}{\inf \spec({\tilde{\D}}^2)}\:. \]
The theorem then follows by using a positivity argument for the
Witten operator and a similar operator built up of the corresponding
wave functions~$\delta \psi_i$.

\subsection{Curvature Estimates Involving the Lowest Eigenvalue on a Conformal Compactification}
\label{secce2}
We now outline how Theorem~\ref{thmweighted} can be used to satisfy the volume
bound~\eqref{volume} in Theorem~\ref{thmg1}. For simplicity, we choose~$t_0=1/4$ and~$t_1=1/2$.
The main step is to prove that choosing the radius
\[ r_1 \::=\: c(n) \,\sigma \left(\sigma \inf \spec |\tilde{D}| \right)^{-\frac{1}{n-1}} \:, \]
the Witten spinor is bounded from below by
\beq \label{lower}
|\psi(x)| \geq \frac{1}{2} \qquad \text{for all~$x \in M \setminus K$ with~$|x| > r_1$}\:.
\eeq
This is achieved by combining elliptic estimates in the spherical cap with spectral estimates
for~${\mathcal{D}}^2$. Then the inequality~\eqref{lower} 
allows us to estimate the volume difference in~\eqref{volume} by
\[ \begin{split}
V \Big( \frac{1}{2} \Big) - V \Big( \frac{1}{4} \Big)
&\leq 16 \int_{\Omega(1/2)} |\psi|^2\: d\mu_M
\leq 16 \int_{B_{r_1}(0)} |\psi|^2\: d\mu_M\\
&\leq 16\int_K\abs{\psi}^2\: d\mu_M +\int_{\set{x\in M\setminus K \;\text{with}\; \abs{x}>r_1}} \abs{\psi}^2\: d\mu_M \:.
\end{split} \]
Using the upper bound~\eqref{es2}, we obtain
\[
V \Big( \frac{1}{2} \Big) - V \Big( \frac{1}{4} \Big)
\leq \int_K |\psi|^2\: d\mu_M \:+\: \mu \Big(
\{x \in M \setminus K \;\text{with}\; |x| \leq r_1 \} \Big) . \]
The first summand can be estimated by Theorem~\ref{thmweighted}, whereas the
second summand can be bounded by the volume of a Euclidean ball of radius~$r_1$.
This method gives the following results (see~\cite[Theorems~1.4 and~4.5]{level}).

\begin{Thm} \label{thmg2}
Let~$(M^n, g)$, $n \geq 3$, be a complete manifold of non-negative
scalar curvature such that~$M \setminus K$ is isometric to
the Schwarzschild geometry.
Then there is an open set~$\Omega \subset M$ with the following
properties. The $(n-1)$-dimensional Hausdorff measure~$\mu_{n-1}$ of the
boundary of~$\Omega$ is bounded by
\beq \label{surf}
\mu_{n-1}(\partial \Omega) \leq c_0(n)\, \sqrt{m}\;
\frac{\left( \rho+m^{\frac{1}{n-2}} \right)^{\frac{n}{2}}}
{\sigma\, \inf \spec |\tilde{\D}|} \:.
\eeq
On the set~$M \setminus \Omega$, the Riemann tensor satisfies the inequality
\[ \int_{M \setminus \Omega} |R|^2 \leq
m \:c_1(n)\: \sup_M |R|  +
\sqrt{m}\:c_2(n)\: \|\nabla R\|_{L^2(M)} \:. \]
\end{Thm}
Note that this theorem involves the surface area of the exceptional set~\eqref{surf}.
The geometry of $K$ enters the estimate only via the smallest eigenvalue
of the Dirac operator on~$\bar{M}$. This is a weaker and apparently more practicable condition than
working with the isoperimetric constant, in particular because eigenvalue estimates 
can be obtained with spectral methods for the Dirac operator on a compact
manifold (see for example~\cite{ginoux}).

\subsection{Results in the Setting with Second Fundamental Form} \label{secsff}
We now return to the setting of general relativity. Thus we again let~$(N^4, \bar{g})$ be
a Lorentzian manifold and~$(M^3, g, h)$ a spacelike hypersurface with induced
Riemannian metric~$g$ and second fundamental form~$h$. Now asymptotic flatness
involves in addition to~\eqref{gdecay} decay assumptions for the second fundamental form,
\[ (\Phi_* h)_{ij} = \O(r^{-2}) \;,\qquad
\partial_k (\Phi_* h)_{ij} = \O(r^{-3}) \;. \]
Total energy and momentum are defined by
\begin{align*}
E &= \frac{1}{16 \pi} \lim_{R \to \infty} \sum_{i,j=1}^3
\int_{S_R} (\partial_j (\Phi_* g)_{ij} - \partial_i (\Phi_* g)_{jj}) \:d\Omega^i \\
P_k &= \frac{1}{8 \pi} \lim_{R
\to \infty} \sum_{i=1}^3\int_{S_R} ((\Phi_* h)_{ki} - \sum_{j=1}^3
\delta_{ki} \:(\Phi_* h)_{jj}) \:d\Omega^i \:.
\end{align*}
The spinor proof of the positive mass theorem as outlined in~\eqref{LW}--\eqref{ibp}
works similarly in the case with second fundamental form, if~${\mathcal{D}}$ is
replaced by the so-called {\em{hypersurface Dirac operator}}, which uses the spin
connection~$\bar{\nabla}$ of the ambient space-time~$N^4$, but acts only in directions
tangential to the hypersurface~$M^3$ (see~\cite{witten, parker+taubes}). The
Lichnerowicz-Weitzenb\"ock formula becomes
\[ {\mathcal{D}}^2 = \bar{\nabla}^*_i \bar{\nabla}^i + {\mathfrak{R}}\:, \]
where now the dominant energy condition ensures that~${\mathfrak{R}}$ is a positive semi-definite
multiplication operator on the spinors. The existence of a solution of the hypersurface Dirac equation
with constant boundary values in the asymptotic end is proved in~\cite{parker+taubes}.
The integration-by-parts argument~\eqref{div} gives in analogy to~\eqref{npsi} the inequality
\[ \int_M |\psi|^2 d\mu_M \leq 4 \pi \left( E + \langle \psi_0, P \cdots \psi_0 \rangle \right) \]
Choosing~$\psi_0$ appropriately, one gets the {\em{positive energy theorem}} $E-|P| \geq 0$.

We now outline the method for deriving curvature estimates (for details see~\cite{kraus}).
As the space-time dimension is larger than three, 
we again need to work with the spinor operator.
Then one can derive an identity similar to~\eqref{nes}, but additional terms involving~$h$
arise. Moreover, $|R|^2$ is to be replaced by the norm of all components of the Riemann tensor
which are determined by the Gauss-Codazzi equations,
\[ |\bar{R}_M|^2 = \sum\limits^3_{i,j=1}\sum\limits^3_{\alpha,\beta=0}(\bar{R}_{ij\alpha\beta})^2 \]
(where the sums run over orthonormal or pseudo-orthonormal frames).
The presence of the second fundamental form leads to the
difficulty that the function $|\psi|^2$ is no longer subharmonic,
making it impossible to estimate the norm of the spinor with the
maximum principle. In order get around this difficulty, we first
construct a barrier function $F$, which is a solution of a
suitable Poisson equation. We then derive Sobolev estimates for
$F$, and these finally give us control of $\||\psi|^2-1\|_{L^6(M)}$.
This leads to the following result (see~\cite[Theorem~1.3]{kraus}).

\begin{Thm}
We choose $L \geq 3$ such that
\[ (L^\alpha -1)^2 \;\geq\; C\: \frac{4\pi E+\|h\|_2}{k^2\: (k+24\: \|h\|_3)^2} \: \||h|^2 + |\nabla h|\|_3 \]
where
\[ \alpha \;=\; \left( 1 + 24\: \frac{\|h\|_3}{k} \right)^{-1} . \;\]
Then there are numerical constants~$c_1, \ldots, c_4$ and
a set $\Omega \subset M$ with measure bounded by
\[ \mu(\Omega) \;\leq\; c_1\: \frac{L^6}{k^2} \:(4\pi E + \|h\|^2_2) \]
such that on $M \setminus \Omega$ the following inequality holds,
\begin{eqnarray*}
\lefteqn{ \int_{M \setminus \Omega} |\overline{R}_M|^2\:d\mu_M \;\leq\;
c_2\: \sup_{M}\left( |h|+ (|R|+|h|^2+|\overline{\nabla}h| \right) E } \\
&&+c_3\:L \:\sup_{M} \left( |\overline{\nabla}\overline{R}_M|+|h||\overline{R}_M| \right) \sqrt{E} \\
&&+c_4\:\frac{\sqrt{L+1}}{k}
\: \sqrt{\||h|^2+|\nabla h|\|_{6/5}} \; \left\||\overline{\nabla}\overline{R}_M|+|h||\overline{R}_M|
\right\|_{5/12} \;\sqrt{E}\;.
\end{eqnarray*}
\end{Thm} \noindent
This theorem is the analog of Theorem~\ref{thm1} for a spacelike hypersurface~$(M^3, g, h)$
of a Lorentzian manifold~$N^4$. Unfortunately, the second fundamental form enters the
theorem in a rather complicated way. It is conceivable that the theorem could be simplified
by improving our method of proof.

The results described in Sections~\ref{seclevel}--\ref{secce2} in the Riemannian setting
so far have not been worked
out in the setting with second fundamental form. Many results could be extended.
However, with the present methods, the proofs and the statements of the results
would be rather involved.

\subsection{Outlook} \label{secoutlook}
We now give a brief outlook on open problems and outline possible directions for future research.
The following problems seem interesting and promising; they have not yet been
studied by us only due to other obligations.
\begin{itemize}
\item Explore the {\em{convexity}} of~$F$: In~\cite[Section~1]{level},
it was shown and briefly discussed that the functional~$F$, \eqref{Fdef}, is convex. However,
the geometric meaning of this convexity has not yet been analyzed. It also seems promising
to search for potential applications.
\item {\em{Extend the weighted $L^2$-estimates}} to more general asymptotically flat
manifolds: The weighted $L^2$-estimates of the Witten spinor~\cite{weighted} were worked out under
the assumption that the manifold is asymptotically Schwarzschild. As a consequence, we
could arrange that the point compactification of the asymptotic end was isometric to a spherical
cap (see Figure~\ref{fig1}), simplifying the elliptic estimates considerably
(see~\cite[Section~5]{weighted}). However, our methods also seem to apply 
to more general asymptotically flat manifolds, possibly with more general compactifications.
\end{itemize}
Moreover, it seems a challenging problem to extend the results outlined in
Sections~\ref{seclevel}--\ref{secce2} to the {\em{setting with second fundamental form}}.
As mentioned at the end of Section~\ref{secsff}, the main difficulty is to improve
our methods so as to obtain simple and clean results.

Our long-term goal is to study the limiting behavior of the manifold as total energy and momentum
tend to zero. Thus, stating the problem for simplicity in the Riemannian setting, we consider a
sequence~$(M_\ell, g_\ell)$ of asymptotically flat manifolds with~$m_\ell \searrow 0$.
In order to get better control of the global geometry, one could make the further assumptions that
the manifolds are all asymptotically Schwarzschild and that the Dirac operators
on the conformal compactifications satisfy the uniform spectral bound
\[ \inf \spec |{\mathcal{D}}_\ell| \geq \varepsilon \qquad \text{for all~$\ell$}\:. \]
Then one could hope that after cutting out exceptional sets~$\Omega_\ell$ of
small surface area~\eqref{surf}, the manifolds~$M_\ell \setminus \Omega_\ell$
converge to flat~$\R^n$ in a suitable sense, for example in a Gromov-Hausdorff sense.
In our attempts to prove results in this direction, we faced the difficulty that convergence
can be established only in suitable charts. Thus on~$M_\ell \setminus \Omega_\ell$
one would like to choose suitable canonical charts, in which the metrics~$g^\ell_{ij}$ converge to the
flat metric~$\delta_{ij}$. Unfortunately, the chart~\eqref{gdecay} is defined only in the asymptotic
end, and thus it would be necessary to extend this chart to~$M_\ell \setminus \Omega_\ell$.
As an alternative, one could hope that the vector fields associated to the Witten spinors~$\psi_i$
form a suitable frame of the tangent bundle. However, it seems difficult to get global control
of this frame. As another alternative, we tried to construct orthonormal
frames~$(e_i)$ by minimizing a corresponding Dirichlet energy,
\[ \int_M \sum_{i=1}^n |\nabla e_i|^2 d\mu_M \rightarrow \text{min} \:. \]
Unfortunately, it seems difficult to rule out that the corresponding minimizer has singularities.
These difficulties were our main obstacle for making substantial progress towards a
proof of Gromov-Hausdorff convergence. But once the problem of choosing a canonical chart
is settled, the limiting behavior of sequences of asymptotically flat manifolds could be attacked.


\section{Minkowski Embeddability of Hypersurfaces in Flat Space-Times} \label{SecTwo}

The positive mass theorem makes two statements on the energy $E$ and the momentum $P$ (see Section \ref{secsff} above) of an asymptotically flat spacelike hypersurface $M$ of a Lorentzian manifold $(\bar{M},\bar{g})$ which satisfies the dominant energy condition (see Section \ref{secdomenergy} below) at every point of $M$. The first statement is that the inequality $E\geq\abs{P}$ holds. The second statement is that if $E=\abs{P}$ holds, then the Riemann tensor of $\bar{g}$ vanishes at every point of $M$. This latter ``rigidity statement'' has been proved by Parker--Taubes \cite{parker+taubes} in the case when $M$ admits a spin structure --- and under the assumption that $M$ is $3$-dimensional, but the argument generalizes to higher dimensions. (The original proof of Witten \cite{witten} deduced the rigidity statement from the stronger assumption that $(\bar{M},\bar{g})$ satisfies the dominant energy condition on a neighborhood of $M$.)

\smallskip
Another proof of the rigidity statement was given by Schoen--Yau \cite{schoen+yau2}, without the spin assumption, but only in the case $\dim M\leq7$. However, Schoen--Yau proved more than Parker--Taubes: they showed that if $E=\abs{P}$ holds, then the Riemannian $n$-manifold $M$ with its given second fundamental form can be embedded isometrically into Minkowski space-time $\R^{n,1}=\R^n\times\R$ as the graph of a function $\R^n\to\R$; in particular, $M$ is diffeomorphic to $\R^n$.

\smallskip
It is natural to ask whether one can decouple the proof of embeddability into Minkowski space-time from the proof of the rigidity statement. That is, when we already know (e.g.\ from the Parker--Taubes proof) that $\bar{g}$ is flat along $M$, can we deduce in a simple way that $M$ with its second fundamental form admits an embedding of the desired form and is in particular diffeomorphic to $\R^n$?

\smallskip
This is indeed possible. The proof works in all dimensions and without topological (e.g.\ spin) conditions. Moreover, it generalises directly to the embeddability of asymptotically hyperbolic hypersurfaces into anti-de Sitter space-time in the rigidity case. This situation is considered in the work of Maerten \cite{Maerten2006}, to which we refer for the definition of the rigidity case in that context. Like Parker--Taubes in the asymptotically flat case, Maerten makes a spin assumption. His proof allows him to obtain an embedding into anti-de Sitter space-time via an explicit construction. Our argument below works differently, without any topological condition.

\smallskip
Stated with minimal assumptions, our result is the following \cite{Nardmannrigidity}: For $c\leq0$, let $\M^{n,1}_c$ denote Minkowski space-time if $c=0$, and anti-de Sitter space-time of curvature $c$ if $c<0$. In each case, $\M^{n,1}_c$ has the form $(\R\times\R^n,-dt^2+g_t)$, where $(g_t)_{t\in\R}$ is a family of Riemannian metrics on $\R^n$. Let $\pr\colon \M^{n,1}_c\to\R^n$ denote the projection $(t,x)\mapsto x$.

\begin{Thm} \label{rigid}
Let $n\geq3$ and $c\in\R_{\leq0}$, let $M$ be a connected $n$-manifold which contains a compact $n$-dimensional submanifold-with-boundary $C$ such that $M\without C$ has a connected component which is simply connected and not relatively compact in $M$. Let $(M,g,K)$ be a complete Riemannian manifold with second fundamental form which satisfies the Gauss and Codazzi equations for constant curvature $c$. Then:
\begin{enumerate}
\item
$(M,g,K)$ admits an isometric embedding $f$ into $\M^{n,1}_c$ such that $\pr\compose f\colon M\to\R^n$ is a diffeomorphism.
\item
When $\tilde{f}$ is an isometric immersion of $(M,g,K)$ into $\M^{n,1}_c$, then there exists an isometry $A\colon\M^{n,1}_c\to\M^{n,1}_c$ with $\tilde{f} = A\compose f$; in particular, $\tilde{f}$ is an embedding.
\end{enumerate}
\end{Thm}

In this theorem, the second fundamental form $K$ is allowed to be a field of symmetric bilinear forms on $M$ with values in an arbitrary (not necessarily trivial) normal bundle of rank $1$. That is, we do not \emph{assume} the normal bundle to be trivial, we get its triviality as a \emph{conclusion} of the theorem (because every line bundle over $\R^n$ is trivial). To understand this, consider the manifold $M=S^1\times\R^{n-1}$ and the flat Riemannian metric $g$ on $M$. It admits an isometric embedding into the flat Lorentzian manifold $\mathfrak{M}\times\R^{n-1}$, where $\mathfrak{M}$ is the Möbius strip, regarded a line bundle over $S^1$ with timelike fibers. The second fundamental form $K$ of this embedding vanishes identically, but the normal bundle is not trivial. $(M,g,K)$ is not a counterexample to Theorem \ref{rigid} because there is no compact $C\subseteq M$ such that $M\without C$ has a simply connected not relatively compact connected component. Replacing $\mathfrak{M}$ by the trivial line bundle over $S^1$ shows that the simply-connectedness assumption in Theorem \ref{rigid} is also needed when the normal bundle is trivial.

\medskip
Let us sketch the proof of Theorem \ref{rigid}. Much of the necessary work is already contained in the semi-Riemannian version of the fundamental theorem of hypersurface theory due to Bär--Gauduchon--Moroianu \cite[Section 7]{BGM}:
\begin{Thm}[Bär--Gauduchon--Moroianu] \label{bgm}
Let $c\in\R$, let $(M,g,K)$ be a Riemannian manifold with second fundamental form which satisfies the Gauss and Codazzi equations for constant curvature $c$. Assume that $M$ is simply connected. Then $(M,g,K)$ admits an isometric immersion into $\M^{n,1}_c$. When $f_0,f_1$ are isometric immersions of $(M,g,K)$ into $\M^{n,1}_c$, then there exists an isometry $A\colon\M^{n,1}_c\to\M^{n,1}_c$ with $f_1 = A\compose f_0$.
\end{Thm}

Recall that a map $f\colon M\to N$ to a Lorentzian manifold $(N,h)$ is \emph{spacelike} iff for every $x\in M$ the image of $T_xf \colon T_xM\to T_{f(x)}N$ is spacelike. A spacelike map $f\colon(M,g)\to(N,h)$ from a Riemannian manifold to a Lorentzian manifold is \emph{long} iff for every interval $I\subseteq\R$ and every smooth path $w\colon I\to M$, the $g$-length of $w$ is finite if the $h$-length of $f\compose w$ is finite. For example, every spacelike isometric immersion is long. The second ingredient for the proof of Theorem \ref{rigid} is the following fact:

\begin{Prp} \label{rigidityprp}
Let $(M,g)$ be a nonempty connected complete Riemannian $n$-mani\-fold, let $f\colon (M,g)\to\M^{n,1}_c$ be a spacelike long immersion into Minkowski space. Then $f\colon M\to\M^{n,1}_c$ is a smooth embedding, and $\pr\compose f\colon M\to\R^n$ is a diffeomorphism.
\end{Prp}

The idea of the proof of \ref{rigidityprp} is as follows. Let us call a map $\phi\colon M\to\R^n$ a \emph{quasicovering} iff it is a local embedding and for all paths $\gamma\colon [0,1]\to\R^n$ and $\tilde{\gamma}\colon[0,1[\to M$ with $\phi\compose\tilde{\gamma} = \gamma| [0,1[$, there exists an extension of $\tilde{\gamma}$ to a path $[0,1]\to M$. One can show that every quasicovering $\phi\colon M\to\R^n$ is a diffeomorphism. This is done in the same way in which one proves the well-known fact that every covering map $M\to\R^n$ is a diffeomorphism (because $\R^n$ is simply connected and $M$ is nonempty and connected).

\smallskip
Now one verifies that $\pr\compose f$ is a quasicovering: It is an immersion (thus a local embedding) because $f$ is a spacelike immersion. For the extension property of a quasicovering, one notes that $f\compose\tilde{\gamma}$ has finite length because $\pr\compose f\compose\tilde{\gamma} = \gamma|[0,1[$ has finite length. Since $f$ is long, $\tilde{\gamma}$ has finite length. Completeness implies that $\tilde{\gamma}$ can be extended to $[0,1]$. Thus $\pr\compose f$ is a quasicovering. (In contrast, it is difficult to show directly that $\pr\compose f$ is a covering map.)

\smallskip
Hence $\pr\compose f$ is a diffeomorphism. Since every proper injective immersion is an embedding, so is $f$. This completes the proof of Proposition \ref{rigidityprp}. (Cf.\ \cite{Nardmannrigidity} for details.)

\medskip
Theorem \ref{rigid} is now easy to prove: We pull back $g$ and $K$ by the universal covering map $p\colon\tilde{M}\to M$ and apply Theorem \ref{bgm}. To the resulting spacelike isometric immersion $\tilde{M}\to\M^{n,1}_c$ we apply Proposition \ref{rigidityprp}. This shows that $\tilde{M}$ is diffeomorphic to $\R^n$. For a connected component $U$ of $M\without C$ as in the statement of Theorem \ref{rigid}, a simple topological argument shows that the covering $p|\,{p^{-1}(U)}\colon p^{-1}(U)\to U$ has only one sheet. Thus $p$ is a diffeomorphism. Now all statements of Theorem \ref{rigid} follow immediately.


\section{Spacelike Foliations and the Dominant Energy Condition} \label{SecThree}

\subsection{Pseudo-Riemannian Manifolds without Spacelike Foliations}

When Lo\-rentzian manifolds are considered in general relativity, it is often assumed that they have nice causality properties like stable causality or even global hyperbolicity. Such manifolds admit a smooth real-valued function with timelike gradient \cite{BernalSanchez2005} and thus a spacelike foliation of codimension $1$, by the level sets of the function. Let us call spacelike foliations of codimension $1$ on a Lorentzian manifold \emph{space foliations} for simplicity. A few years ago, Christian Bär asked us whether \emph{every} Lorentzian manifold admits a space foliation.

\smallskip
The answer is not obvious, for the following reasons. First, clearly every point in a Lorentzian manifold has an open neighborhood which admits a space foliation.

\smallskip
Second, the tangent bundle of every semi-Riemannian manifold has an orthogonal decomposition $V\oplus H$ into a timelike sub vector bundle $V$ and a spacelike sub vector bundle $H$. (At every point of an $n$-dimensional manifold $M$ which is equipped with a semi-Riemannian metric of index $q$, the choice of a time/space splitting corresponds to a point in the contractible space $\OO(n)/(\OO(q)\times\OO(n-q))$. Thus a global time/space splitting of the tangent bundle $TM$ exists if a certain fiber bundle over $M$ with contractible fibers admits a smooth section. Obstruction theory tells us that such a section exists for every manifold and metric.)

\smallskip
The question is therefore whether the spacelike bundle $H$ can always be chosen \emph{integrable}, i.e.\ tangent to a foliation. Since \emph{every} sub vector bundle of rank $1$ of a tangent bundle is integrable, it is clear that every $2$-dimensional Lorentzian manifold admits a space foliation. (There are many quite complicated examples of Lorentzian $2$-manifolds, because every noncompact connected smooth $2$-manifold admits a Lorentz\-ian metric.)

\smallskip
Third, a theorem of W.~Thurston says that every connected component of the space of $(n-1)$-plane distributions on an $n$-manifold $M$ contains an integrable distribution \cite{Thurston1976}. Here we use the word \emph{distribution} in the differential-topological sense: a \emph{$k$-plane distribution} on a manifold $M$ is a sub vector bundle of rank $k$ of $TM$. Distributions can be viewed as sections in the bundle $\Gr_k(TM)\to M$ whose fiber over $x$ is the Grassmann manifold $\Gr_k(T_xM)$ of $k$-dimensional sub vector spaces of $T_xM$. Connected components of the set of $k$-plane distributions on $M$ are considered with respect to the compact-open topology on the space of sections in $\Gr_k(TM)\to M$. In contrast to the situation for Riemannian metrics, the space of Lorentzian metrics on a given manifold can be empty or have several connected components.

\smallskip
Thurston's theorem implies that every connected component of the space of Lo\-rentz\-ian metrics on a manifold contains metrics which admit space foliations: The set of connected components of the space of $(n-q)$-plane distributions on an $n$-manifold $M$ is in canonical bijective correspondence to the set of connected components of the space of semi-Riemannian metrics of index $q$ on $M$. The correspondence maps the connected component of each distribution $H$ to the connected component of a metric which makes $H$ spacelike.

\smallskip
These facts show that there are no \emph{topological} obstructions to the existence of space foliations on Lorentzian manifolds (in contrast to the situation on semi-Riemannian manifolds of higher index: the analog of Thurston's theorem is in general false for distributions of codimension $\geq2$). Nevertheless, the answer to Bär's question is negative. Counterexamples exist even on topologically trivial manifolds like $\R^n$ (see \cite[Theorem 0.1]{Nardmann2007}):

\begin{Thm} \label{foliationless}
Let $(M,g)$ be an $n$-dimensional pseudo-Riemannian manifold of index $q\in\set{1,\dots,n-2}$ (e.g.\ a Lorentzian manifold of dimension $n\geq3$). Let $A\neq M$ be a closed subset of $M$. Then there exists a metric $g'$ of index $q$ on $M$ such that
\begin{enumerate}
\item
$g=g'$ on $A$;
\item
every $g$-timelike vector in $TM$ is $g'$-timelike;
\item
$M\without A$ does not admit any codimension-$q$ foliation none of whose tangent vectors is $g'$-timelike; in particular, $(M,g')$ does not admit any space foliation.
\end{enumerate}
\end{Thm}

Here and in the following, our conventions are such that $v\in TM$ is $g$-spacelike resp.\ $g$-timelike resp.\ $g$-causal iff $g(v,v)>0$ resp.\ $g(v,v)<0$ resp.\ $g(v,v)\leq0$; such that the index of a metric is the maximal dimension of timelike sub vector spaces of tangent spaces; and such that Lorentzian metrics have index $1$.

\smallskip
The idea of the proof of Theorem \ref{foliationless} is simple: We choose a $g$-spacelike $(n-q)$-plane distribution $H$ on $M$ and modify it on $M\without A$ in such a way that the new distribution $H'$ is not integrable on $M\without A$ but still $g$-spacelike; this is possible because $2\leq n-q\leq n-1$. We construct a sequence $(g_k)_{k\in\N}$ of semi-Riemannian metrics of index $q$ on $M$ with $g_0=g$ such that each $g_k$ is equal to $g$ on $A$; and such that on some compact ball $B$ in $M\without A$, the $g_k$-lightcones become wider and wider as $k$ tends to $\infty$, and the $g_k$-spacelike regions ``converge'' to $H'$ as they become smaller with increasing $k$ (cf.\ Figure \ref{prooffigure}).

\smallskip
We claim that for sufficiently large $k$, the restriction of $g_k$ to $B$ does not admit a space foliation. Otherwise we would obtain a sequence $(H_k)_{k\in\N}$ of integrable distributions on $B$ such that every $H_k$ is $g_k$-spacelike. By our construction of the metrics $g_k$, this sequence would converge in the $C^0$-topology to the nonintegrable distribution $H'$. But $C^0$-limits of integrable distributions are always integrable; cf.\ \cite{Nardmann2007} for details (or \cite{Varela1976} for a slightly different proof sketch in the case $n-q=n-1$). This contradiction proves our claim. Now the proof of Theorem \ref{foliationless} is complete: we can take $g'=g_k$ for any sufficiently large $k$.

\begin{figure}[hbtp]
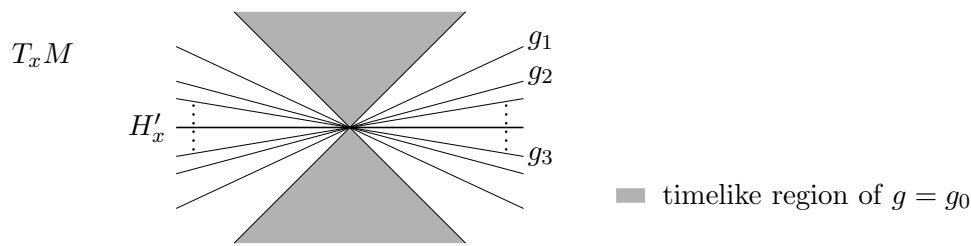

\vspace*{6ex}
\[
\psset{unit=2em,nodesep=2em,linewidth=0.2pt,origin={0,0},dotsep=2em}
\psset{fillcolor=grey}
\pspolygon*[linecolor=grey](-2.5,0.0)(-4.5,-2.0)(-0.5,-2.0)
\pspolygon*[linecolor=grey](-2.5,0.0)(-4.5,2.0)(-0.5,2.0)
\psframe*[linecolor=grey](2.1,-1.3)(2.6,-1.05)
\psset{fillcolor=grey}
\rput[l](2.9,-1.19){\Rnode{s}{\text{timelike region of $g=g_0$}}}
\psline(-5.5,1.4)(0.5,-1.4)
\psline(-5.5,-1.4)(0.5,1.4)
\psline(-5.5,0.8)(0.5,-0.8)
\psline(-5.5,-0.8)(0.5,0.8)
\psline(-5.5,0.5)(0.5,-0.5)
\psline(-5.5,-0.5)(0.5,0.5)
\psline[linewidth=0.6pt](-5.5,0.0)(0.5,0.0)
\psline(-4.5,-2.0)(-0.5,2.0)
\psline(-4.5,2.0)(-0.5,-2.0)
\psline[linestyle=dotted,dotsep=0.1,linewidth=1pt](0.2,0.05)(0.2,0.4)
\psline[linestyle=dotted,dotsep=0.1,linewidth=1pt](0.2,-0.05)(0.2,-0.4)
\psline[linestyle=dotted,dotsep=0.1,linewidth=1pt](-5.2,0.05)(-5.2,0.4)
\psline[linestyle=dotted,dotsep=0.1,linewidth=1pt](-5.2,-0.05)(-5.2,-0.4)
\rput(0.8,1.5){\Rnode{g1}{g_1}}
\rput(0.8,0.9){\Rnode{g1}{g_2}}
\rput(0.8,-0.5){\Rnode{g1}{g_3}}
\rput(-6.0,0.0){\Rnode{H}{H'_x}}
\rput(-7.8,1.25){\Rnode{T}{T_xM}}
\]
\vspace*{6ex}
\caption{The lightcones at a point $x\in B$ of the metrics $g_k$ in the proof of Theorem \ref{foliationless}.} \label{prooffigure}
\end{figure}

\subsection{Existence of Dominant Energy Metrics} \label{secdomenergy}

One might be tempted to regard Lorentzian metrics without space foliations as objects of little physical relevance, things which could only serve as examples of the strange phenomena that occur when the standard causality assumptions in general relativity are dropped. But there is another side of the medal: certain physically \emph{desirable} properties --- which from the geometric view\-point are conditions on the Ricci curvature --- can in general be satisfied \emph{only} by Lorentzian metrics without space foliations.

\smallskip
This holds in particular for the dominant energy condition, which plays an important role in the positive energy theorem (cf.\ sections \ref{SecOne} and \ref{SecTwo} above). In the following discussion we will use the version with arbitrary cosmological constant:

\begin{Def}
Let $(M,g)$ be a Lorentzian manifold, let $\Lambda\in\R$. The \emph{energy-momentum tensor of $(M,g)$ with respect to (the cosmological constant) $\Lambda$} is the symmetric $(0,2)$-tensor $T= \Ric -\frac{1}{2}\scal g +\Lambda g$; here $\scal$ and $\Ric$ are the scalar and Ricci curvatures of $g$, respectively. (This means that we interpret Einstein's field equation as the definition of the energy-momentum tensor when $g$ and $\Lambda$ are given.)

\smallskip
\emph{$(M,g)$ satisfies the dominant energy condition with respect to $\Lambda$} iff for every $x\in M$ and every $g$-timelike vector $v\in T_xM$, the vector $-T^a{}_bv^b$ lies in the closure of the connected component of $\set{u\in T_xM \suchthat g(u,u)<0}$ which contains $v$. (The abstract index notation $-T^a{}_bv^b$ describes the vector which is the $g$-dual of the linear form $T(.,v)$ on $T_xM$.)

\smallskip
In other words, $(M,g)$ satisfies the dominant energy condition with respect to $\Lambda$ iff every $g$-timelike vector $v\in TM$ satisfies $T(v,v)\geq0$ and $g(w,w)\leq0$, where $w^a = -T^a{}_bv^b$.
\end{Def}

In general relativity, every physically reasonable space-time metric $g$ should satisfy the dominant energy condition: Timelike vectors $v$ are tangents to observer wordlines. Every observer should see a nonnegative energy density at each space-time point she passes through; that is expressed by the condition $T(v,v)\geq0$. And she should see that matter does not move faster than light; that is what $g(w,w)\leq0$ means ($w$ is the momentum density observed by $v$).

\medskip
In view of the physical importance of the dominant energy condition, a natural geometric question arises: \emph{For given cosmological constant $\Lambda$, which manifolds $M$ admit a Lorentzian metric that satisfies the dominant energy condition with respect to $\Lambda$?} A trivially necessary condition is that $M$ admits a Lorentzian metric at all, but are there other conditions? (Note that every noncompact connected manifold admits a Lorentzian metric.)

\smallskip
Riemannian geometry offers many nonexistence results for metrics of nonnegative scalar curvature, most notably the positive energy theorem and obstructions to Riemannian metrics of nonnegative scalar curvature on closed manifolds. On every spacelike hypersurface $S$ in a Lorentzian manifold which satisfies the dominant energy condition for some $\Lambda$, the Gauss equation yields an inequality $s\geq\ldots$, where $s$ is the scalar curvature of the induced Riemannian metric on $S$ and ``$\ldots$'' depends on $\Lambda$ and the second fundamental form of $S$.

\smallskip
One might therefore suspect that there exist obstructions to the existence of dominant energy metrics. For instance, when $M$ has the form $S^1\times N$ for some closed manifold $N$ which does not admit a Riemannian metric of nonnegative scalar curvature, then it is difficult to satisfy a $\Lambda\geq0$ dominant energy condition with a Lorentzian metric which makes every submanifold $N_t = \set{t}\times N$ spacelike: for every $t\in S^1$, properties of the second fundamental form of $N_t$ would have to compensate the negative scalar curvature of $N_t$ that occurs unavoidably on some subset of $N_t$.

\smallskip
However, one can construct dominant energy metrics in a way that circumvents such problems. In our $S^1\times N$ example, we can even arrange that the vector field $\partial_t$ along the $S^1$ factor becomes timelike (see \cite[Theorem 0.5]{Nardmann2007}):

\begin{Thm} \label{maindom}
Let $(M,g)$ be a connected Lorentzian manifold of dimension $n\geq4$, let $K$ be a compact subset of $M$, let $\Lambda\in\R$. If $n=4$, assume that $(M,g)$ is time- and space-orientable, and that either $M$ is noncompact, or compact with intersection form signature divisible by $4$. Then there exists a Lorentzian metric $g'$ on $M$ such that
\begin{enumerate}
\item
every $g$-causal vector in $TM$ is $g'$-timelike;
\item
$g'$ satisfies on $K$ the dominant energy condition with cosmological constant $\Lambda$;
\item
$(M,g')$ does not admit a space foliation.
\end{enumerate}
\end{Thm}

Recall the definition of the intersection form signature of a closed oriented $4$-manifold $M$: On the de Rham cohomology $\R$-vector space $V=H^2_{\text{dR}}(M)$, we have a symmetric nondegenerate bilinear form $\omega\colon V\times V\to\R$ given by $\omega([\alpha],[\beta]) = \int_M\alpha\wedge\beta$. Diagonalization of $\omega$ yields a diagonal matrix with $p$ positive and $q$ negative entries. The intersection form signature of $M$ is $p-q$. When we reverse the orientation of $M$, then $\omega$ and the signature change their signs. Thus the condition that the signature be divisible by $4$ is well-defined for a closed orientable $4$-manifold. If a closed $4$-manifold admits a Lorentzian metric (this happens iff the Euler characteristic vanishes) and is orientable, then its signature is automatically even.

\smallskip
Theorem \ref{maindom} generalizes to dimension $3$ when one assumes that $M$ is orientable \cite[Theorem 8.6]{Nardmann2007}. However, property (1) must then be replaced by the weaker property that $g'$ lies in the same connected component of the space of Lorentzian metrics as $g$.

\medskip
Let us consider again the case where $M$ has the form $S^1\times N$ for some closed manifold $N$ which does not admit a Riemannian metric of nonnegative scalar curvature --- the connected sum of two $3$-tori, say, for which $S^1\times N$ is orientable and has intersection form signature $0$.
Then we can choose any Riemannian metric $g_N$ on $N$ and consider the product Lorentzian metric $g = -dt^2\oplus g_N$ on $M$. This $g$ does certainly not satisfy the dominant energy condition for any $\Lambda\geq0$. It is time- and space-orientable and makes the vector field $\partial_t$ timelike. For every $\Lambda$, Theorem \ref{maindom} applied to $K=M$ yields a metric $g'$ on $M$ which satisfies the dominant energy condition with respect to $\Lambda$ and still makes $\partial_t$ timelike.

\smallskip
We will sketch the proof of Theorem \ref{maindom} in a moment, but already at this point property (3) provides a hint how the difficulties arising from Riemannian nonnegative scalar curvature obstructions can be avoided: the foliation given by the leaves $N_t$ will not be spacelike for the metric $g'$. (Probably no compact hypersurface of $M$ will be $g'$-spacelike, but that is not obvious from the proof.)

\smallskip
For the proof of Theorem \ref{maindom}, we need a pointwise measure of the nonintegrability of a distribution \cite[p.~176]{Thurston}:
\begin{Def}
Let $H$ be a $k$-plane distribution on a manifold $M$. The \emph{twistedness $\Tw_H$ of $H$} is a $TM/H$-valued $2$-form on $M$ (i.e.\ a section in $\bigwedge^2(H^\ast)\otimes(TM/H)$) which is defined as follows. Let $[.,.]_x$ denote the Lie bracket of vector fields on $M$ evaluated in a point $x\in M$, and let $\pi\colon TM\to TM/H$ denote the projection. For $x\in M$ and $v,w\in H_x$, we define
\[
\Tw_H(v,w) = \pi([v,w]_x) \;\;,
\]
where on the right-hand side we have extended the vectors $v,w$ to sections in $H$. The definition does not depend on the extension, because it is antisymmetric and $C^\infty(M,\R)$-linear in $v$: $\pi(v)=0$ implies that $\pi([fv,w]) = \pi(f[v,w]-df(w)v) = f\pi([v,w])$ holds for every $f\in C^\infty(M,\R)$.
\end{Def}

By Frobenius' theorem, a distribution $H$ is integrable (i.e.\ tangent to a foliation) if and only if $\Tw_H$ vanishes identically. In contrast to the proof of Theorem \ref{foliationless} above, we need for the proof of Theorem \ref{maindom} distributions $H$ which are not just not integrable, but even pointwise nonintegrable in the sense that $\Tw_H$ does not vanish in any point $x\in M$; i.e., we need that for every $x\in M$, there exist sections $v,w$ in $H$ such that the Lie bracket value $[v,w]_x$ is not contained in $H_x$. Let us call a distribution with this property \emph{twisted}.

\smallskip
The first step in the proof is to choose for the given Lorentzian metric $g$ a spacelike twisted $(n-1)$-plane distribution $H$. In order to find such a distribution, we start with an arbitrary spacelike distribution and prove that it can be approximated in the fine $C^0$-topology by twisted distributions. (For the definition of the \emph{fine}, also known as \emph{Whitney}, \emph{$C^0$-topology}, see \cite[p.~35]{Hirsch}. Because Theorem \ref{maindom} arranges the dominant energy condition only on a compact set $K$, it would suffice here to prove that $H$ can be approximated \emph{on $K$} in the fine $C^0$-topology, which is the same as the compact-open topology because $K$ is compact. But we want to emphasize that this first step of the proof works also on noncompact manifolds.)

\smallskip
The proof of this approximation employs M.\ Gromov's h-principle for ample open partial differential relations, also known as the convex integration method; cf.\ \cite{GromovPDR} or \cite{Spring}. In dimensions $n\geq5$, one can even use R.\ Thom's jet transversality theorem \cite[Theorem 2.3.2]{EliashbergMishachev}, which shows that twisted distributions lie not only $C^0$-dense but even $C^\infty$-dense in the space of distributions. The situation in dimension $4$ is more subtle and requires the additional assumptions in Theorem \ref{maindom}. For instance, if the bundles $H$ and $TM/H$ are orientable and the intersection form signature of $M$ is congruent to $2$ modulo $4$, then the connected component of $H$ does not contain any twisted distribution, not even far away from $H$. Theorems of Hirzebruch--Hopf \cite{HirzebruchHopf} and Donaldson \cite{Donaldson1987} are applied to solve the problem when the manifold is noncompact or the signature is divisible by $4$. For details of these differential-topological considerations see Chapter 5 of the second author's PhD thesis \cite{Nardmann2004}.

\smallskip
Since every distribution which is sufficiently $C^0$-close to our spacelike start distribution is spacelike as well, we obtain a spacelike twisted distribution $H$, as desired. Let $V$ be the (timelike) $g$-orthogonal complement of $H$. For every $f\in C^\infty(M,\R_{>0})$, we can now consider the Lorentzian metric $g'$ on $M$ which is given by
\begin{equation} \label{gprime}
g'(v_0+h_0,v_1+h_1) = \tfrac{1}{f^2}g(v_0,v_1) +g(h_0,h_1)
\end{equation}
for all $x\in M$ and $h_0,h_1\in H_x$ and $v_0,v_1\in V_x$. Using similar arguments as in the proof of Theorem \ref{foliationless}, we see that there exists a constant $\varepsilon_0>0$ such that whenever $f\leq\varepsilon_0$, then the metric $g'$ has the properties (1) and (3) of Theorem \ref{maindom}.

\smallskip
Now we have to compute the Ricci tensor of $g'$ and check whether $g'$ satisfies the dominant energy condition. Let $x\in M$. If $(v,w)\in H_x\times H_x$, then
\begin{equation} \label{Ricformula} \begin{split}
\Ric_{g'}(v,w)
&= \Ric_g(v,w)
+\tfrac{1}{f}\Hess_g\!f\,(v,w)
-\tfrac{2}{f^2}df(v)df(w)\\
&\mspace{20mu}
+\tfrac{1}{f}\divergence^V_g(w)df(v)
+\tfrac{1}{f}\divergence^V_g(v)df(w)
+\tfrac{1+f^2}{2f}\Symst{H}{g}(v,w,a)df(a)\\
&\mspace{20mu}
-(1-f^2)\Phi^H_g(v,w)
-\tfrac{1-f^2}{2f^2}\Twist{H}{g}(v,a,b)\Twist{H}{g}(w,a,b) \;\;.
\end{split} \end{equation}
If $(v,w)\in V_x\times V_x$, then
\[ \begin{split}
f^2\Ric_{g'}(v,w)
&= f^2\Ric_g(v,w)
-f\Hess_g\!f(v,w)
+\Big(\tfrac{1}{f}\laplace^H_g\!f +f\laplace^V_g\!f\Big)g(v,w)\\
&\mspace{20mu}
-\tfrac{2}{f^2}\abs{df}_{g,H}^2g(v,w)
-f\divergence^H_g(w)df(v)
-f\divergence^H_g(v)df(w)\\
&\mspace{20mu}
+\Big(\tfrac{1}{f}\divergence^V_g(a)df(a) +f\divergence^H_g(a)df(a)\Big)g(v,w)
-\tfrac{1+f^2}{2f}\Symst{V}{g}(v,w,a)df(a)\\
&\mspace{20mu}
+(1-f^2)\Phi^V_g(v,w)
+\tfrac{(1-f^2)^2}{4f^2}\Twist{H}{g}(a,b,v)\Twist{H}{g}(a,b,w)
\;\;.
\end{split} \]
If $(v,w)\in V_x\times H_x$, then
\[ \begin{split}
f\Ric_{g'}(v,w)
&= f\Ric_g(v,w)
+\tfrac{3}{2f^2}\Twist{H}{g}(w,a,v)df(a)
-\divergence^H_g(v)df(w)
+\divergence^V_g(w)df(v)\\
&\mspace{20mu}
+\sigma^H_g(v,w,a)df(a)
+\Big(\tfrac{1-f^2}{2f}\tilde{\Theta}^H_g
-\tfrac{f(1-f^2)}{2}\tilde{\Theta}^V_g
+\tfrac{(1-f^2)^2}{4f}\tilde{\Xi}^H_g\Big)(v,w) \;\;.
\end{split} \]
In these formulas, $\divergence_g^?$ are certain $(0,1)$-tensors, $\Phi^?_g, \tilde{\Theta}^?_g, \tilde{\Xi}^H_g$ are $(0,2)$-tensors, and $\Sw^?_g,\sigma_g^H$ are $(0,3)$-tensors, induced by the metric $g$ and the distributions $V,H$. The precise definitions are not relevant for our discussion. $\Hess_gf$ is the $g$-Hessian of $f$, $\laplace_g^Uf$ is the $g$-contraction of the restriction of $\Hess_g$ to the subbundle $U$, and $\abs{df}^2_{g,H}$ is the $g$-contraction of the restriction of $df\otimes df$ to the subbundle $H$. Arguments $a,b$ occur always pairwise in the formulas; they have to be interpreted in the sense of a summation convention, i.e., a $g$-contraction is performed in these tensor indices.

\smallskip
Finally, $\Tw^H_g$ is the $(0,3)$-tensor on $M$ which is defined as follows: We identify $TM/H$ with $V$. Using the projection $\pi_H\colon TM=V\oplus H\to H$, we let
\[
\Tw^H_g(v,w,z) = g\big(\Tw_H(\pi_Hv,\pi_Hw),z\big) \;\;.
\]

\smallskip
Although the formula for $\Ric_{g'}$ is quite complicated, it is relatively easy to see what happens on a compact set $K\subseteq M$ when the function $f$ is a very \emph{small} positive \emph{constant}: All summands containing derivatives of $f$ are zero, and at every point $x\in K$ where $\Tw^H_g$ does not vanish, the summands containing $\Tw^H_g\otimes\Tw^H_g$ dominate the Ricci tensor because their coefficients have $f^2$ in the denominator. In this situation one can determine from the tensor field $\Tw^H_g$ alone whether $g'$ satisfies the dominant energy condition for a given $\Lambda$.

\smallskip
The result is that if $H$ is twisted, then for every $\Lambda$ and every compact subset $K$ of $M$ there exists a constant $c_{K,\Lambda}>0$ such that for every constant $f$ with $0<f<c_{K,\Lambda}$, the metric $g'$ satisfies on $K$ the dominant energy condition with respect to $\Lambda$. This completes the proof of Theorem \ref{maindom}.

\medskip
One can apply an analogous ``stretching'' by a function $f$ as in equation \eqref{gprime} to an arbitrary semi-Riemannian metric $g$ of index $q$ and an arbitrary $g$-spacelike $k$-plane distribution $H$. The effect is similar: For every comparison of $\Ric_{g'}$ with $g'$, the most important contribution to $\Ric_{g'}$ comes from the nonintegrability properties of $H$ when $f$ is a small constant. This principle can also be employed to prove new results in \emph{Riemannian} geometry \cite{Nardmannhprinciple}.

\subsection{Lorentz Cobordisms and Topology Change}

A striking example of the difference between metrics without space foliations and space-foliated metrics occurs in the classical problem of ``topology change'' in general relativity. This problem deals with the question whether the spatial topology of our universe could change as time goes by. It was discussed by several authors in the 1960s and 1970s; cf.\ e.g.\ \cite{Reinhart1963,Geroch1967,Yodzis1,Yodzis2,Tipler1977}. In order to describe it in detail, let us adopt the following terminology (see also Figure \ref{cobordism}):

\begin{Def}
Let $S_0,S_1$ be $(n-1)$-dimensional compact manifolds. A \emph{weak Lorentz cobordism between $S_0$ and $S_1$} is a compact $n$-dimensional Lorentzian man\-ifold-with-boundary $(M,g)$ whose boundary is the disjoint union $S_0\sqcup S_1$, such that $M$ admits a $g$-timelike vector field which is inward-directed on $S_0$ and outward-directed on $S_1$. A \emph{Lorentz cobordism between $S_0$ and $S_1$} is a weak Lorentz cobordism $(M,g)$ between $S_0$ and $S_1$ such that $\mfbd M$ is $g$-spacelike. $S_0$ is [\emph{weakly}] \emph{Lorentz cobordant} to $S_1$ iff there exists a [weak] Lorentz cobordism between $S_0$ and $S_1$.
\end{Def}

\begin{figure}[hbtp]
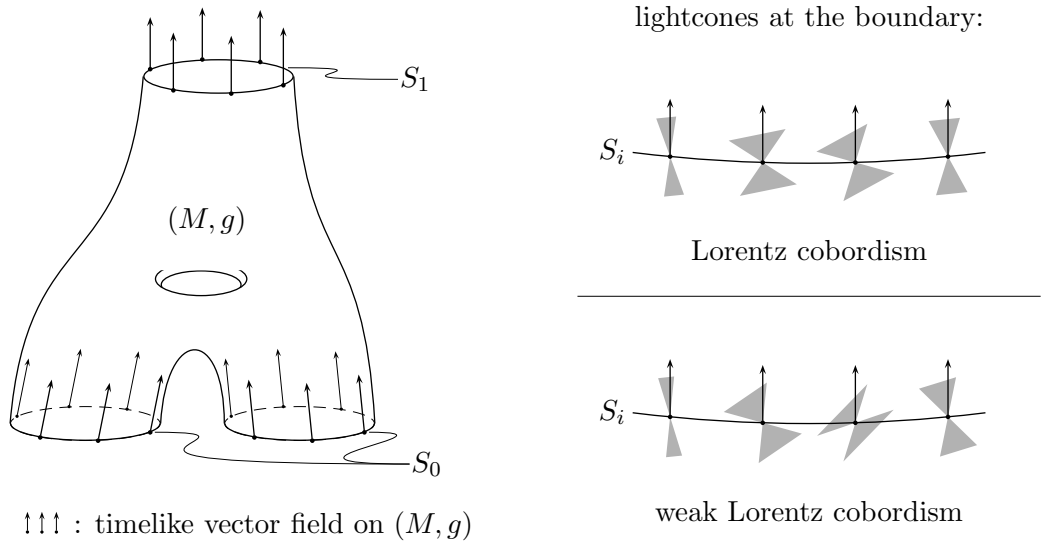

\vspace*{18ex}
\[
\psset{unit=2em,nodesep=2em,linewidth=0.3pt,origin={0,0},dotsep=2em,dotsize=.5pt 2.5}
\psellipse[linestyle=dashed,linecolor=green](-3.8,-3)(1.3,0.3)
\psellipse[linestyle=dashed,linecolor=green](-7.5,-3)(1.3,0.3)
\psellipse[linewidth=0.5pt,linecolor=green](-5.2,3)(1.3,0.3)
\psbezier[linewidth=0.5pt](-8.8,-3)(-8.5,0.5)(-6.8,-0.5)(-6.5,3)
\psbezier[linewidth=0.5pt](-2.5,-3)(-2.7,0.5)(-3.7,-0.5)(-3.9,3)
{\psclip{\pspolygon[linestyle=none](-9,-3)(-9,-6)(-2,-6)(-2,-3)}
\psellipse[linewidth=0.5pt,linecolor=green](-3.8,-3)(1.3,0.3)
\psellipse[linewidth=0.5pt,linecolor=green](-7.5,-3)(1.3,0.3)
\endpsclip}
\psbezier[linewidth=0.5pt](-6.2,-3)(-6.1,-1.3)(-5.1,-1.3)(-5.1,-3)
{\psclip{\pspolygon[linestyle=none](-6.5,-1)(-4,-1)(-4,-0.35)(-6.5,-0.35)}
{\psclip{\psellipse[linewidth=0.5pt](-5.5,-0.5)(0.8,0.3)}
\psellipse[linewidth=0.5pt](-5.5,-0.6)(0.7,0.25)
\endpsclip}
\endpsclip}
\rput(-5.4,0.5){\Rnode{Mg}{\textcolor{blue}{(M,g)}}}
\psbezier[linecolor=green](-4.0,3.15)(-1.8,3.15)(-4.9,2.95)(-2.1,2.95)
\rput(-1.8,2.95){\Rnode{S1}{\textcolor{green}{S_1}}}
\psbezier[linecolor=green](-6.3,-3.15)(-3.8,-3.15)(-8.9,-3.7)(-1.9,-3.7)
\psbezier[linecolor=green](-2.6,-3.15)(-0.8,-3.15)(-4.9,-3.7)(-1.9,-3.7)
\rput(-1.6,-3.7){\Rnode{S0}{\textcolor{green}{S_0}}}
\psline[linecolor=red,linewidth=0.5pt]{*->}(-5.98,2.76)(-5.98,3.76)
\psline[linecolor=red,linewidth=0.5pt]{*->}(-4.08,2.848)(-4.08,3.848)
\psline[linecolor=red,linewidth=0.5pt]{*->}(-4.98,2.704)(-4.98,3.704)
\psline[linecolor=red,linewidth=0.5pt]{*->}(-6.38,3.126)(-6.38,4.026)
\psline[linecolor=red,linewidth=0.5pt]{*->}(-5.48,3.293)(-5.48,4.193)
\psline[linecolor=red,linewidth=0.5pt]{*->}(-4.48,3.250)(-4.48,4.150)
\psline[linecolor=red,linewidth=0.5pt]{*->}(-8.28,-3.24)(-8.08,-2.24)
\psline[linecolor=red,linewidth=0.5pt]{*->}(-6.38,-3.152)(-6.18,-2.152)
\psline[linecolor=red,linewidth=0.5pt]{*->}(-7.28,-3.296)(-7.08,-2.296)
\psline[linecolor=red]{*->}(-8.68,-2.874)(-8.48,-1.874)
\psline[linecolor=red]{*->}(-7.78,-2.707)(-7.58,-1.707)
\psline[linecolor=red]{*->}(-6.78,-2.750)(-6.58,-1.750)
\psline[linecolor=red,linewidth=0.5pt]{*->}(-4.58,-3.24)(-4.68,-2.24)
\psline[linecolor=red,linewidth=0.5pt]{*->}(-2.68,-3.152)(-2.78,-2.152)
\psline[linecolor=red,linewidth=0.5pt]{*->}(-3.58,-3.296)(-3.68,-2.296)
\psline[linecolor=red]{*->}(-4.98,-2.874)(-5.08,-1.874)
\psline[linecolor=red]{*->}(-4.08,-2.707)(-4.18,-1.707)
\psline[linecolor=red]{*->}(-3.08,-2.750)(-3.18,-1.750)
\psline[linecolor=red]{*->}(-8.5,-4.875)(-8.5,-4.525)
\psline[linecolor=red]{*->}(-8.25,-4.875)(-8.25,-4.525)
\psline[linecolor=red]{*->}(-8.0,-4.875)(-8.0,-4.525)
\rput[l](-7.7,-4.76){\Rnode{red}{\textcolor{red}{\text{\normalsize: timelike vector field on $(M,g)$}}}}
\rput(5,4){\textcolor{blue}{\text{\normalsize lightcones at the boundary:}}}
\rput(5,0){\textcolor{blue}{\text{\normalsize Lorentz cobordism}}}
\rput(5,-4.5){\textcolor{blue}{\text{\normalsize weak Lorentz cobordism}}}
\pspolygon*[linecolor=grey](2.705,2.307)(2.360,2.273)(2.840,0.957)(2.495,0.923)
\pspolygon*[linecolor=grey](4.595,2.091)(3.611,1.891)(4.789,1.135)(3.805,0.935)
\pspolygon*[linecolor=grey](6.007,2.182)(5.126,1.700)(6.474,1.326)(5.593,0.844)
\pspolygon*[linecolor=grey](7.607,2.284)(7.064,2.229)(7.736,1.001)(7.193,0.946)
\pspolygon*[linecolor=grey](2.670,-2.186)(2.393,-2.216)(2.807,-3.554)(2.530,-3.582)
\pspolygon*[linecolor=grey](4.270,-2.291)(3.526,-2.800)(4.875,-3.174)(4.130,-3.684)
\pspolygon*[linecolor=grey](6.007,-2.318)(5.138,-3.213)(6.462,-2.761)(5.593,-3.656)
\pspolygon*[linecolor=grey](7.539,-2.199)(6.898,-2.397)(7.902,-3.373)(7.261,-3.571)
\psarc[linecolor=green,linewidth=0.5pt](5,22.0){25}{263}{277}
\rput(1.6,-2.8){\textcolor{green}{S_i}}
\psarc[linecolor=green,linewidth=0.5pt](5,26.5){25}{263}{277}
\rput(1.6,1.7){\textcolor{green}{S_i}}
\psline(1,-0.8)(9,-0.8)
\psline[linecolor=red,linewidth=0.5pt]{*->}(2.6,1.615)(2.6,2.615)
\psline[linecolor=red,linewidth=0.5pt]{*->}(4.2,1.513)(4.2,2.513)
\psline[linecolor=red,linewidth=0.5pt]{*->}(5.8,1.513)(5.8,2.513)
\psline[linecolor=red,linewidth=0.5pt]{*->}(7.4,1.615)(7.4,2.615)
\psline[linecolor=red,linewidth=0.5pt]{*->}(2.6,-2.885)(2.6,-1.885)
\psline[linecolor=red,linewidth=0.5pt]{*->}(4.2,-2.987)(4.2,-1.987)
\psline[linecolor=red,linewidth=0.5pt]{*->}(5.8,-2.987)(5.8,-1.987)
\psline[linecolor=red,linewidth=0.5pt]{*->}(7.4,-2.885)(7.4,-1.885)
\]
\vspace*{20ex}
\caption{A [weak] Lorentz cobordism. The picture on the left is not optimal because the vector field cannot be extended nonvanishingly to the whole cobordism. The viewer is supposed to imagine that it can. (It is impossible to draw a good picture, because nontrivial $2$-dimensional [weak] Lorentz cobordisms do not exist.)} \label{cobordism}
\end{figure}

One can generalize this definition to noncompact manifolds, but for reasons of simplicity and mathematical elegance we will discuss only the compact case here. Moreover, we will concentrate on the dimension which is relevant in general relativity, $n=4$.

\smallskip
[Weak] Lorentz cobordance is an equivalence relation. Clearly two manifolds are Lorentz cobordant if and only if they are weakly Lorentz cobordant. But when we require the cobordism metric $g$ to satisfy the dominant energy condition --- recall that every physically realistic space-time metric should have this property ---, then the two cobordance relations become very different, as we will see.

\smallskip
In 1963, Reinhart \cite{Reinhart1963} proved (using the well-known computations of the unoriented and oriented cobordism rings due to R.~Thom and C.T.C.~Wall) that every two closed $3$-manifolds $S_0,S_1$ are Lorentz cobordant; and that for every two closed \emph{oriented} $3$-manifolds, there exists a Lorentz cobordism $(M,g)$ between $S_0$ and $S_1$ and an orientation of $M$ which turns $M$ into an oriented cobordism in the usual sense (i.e., the orientation of $M$ induces the given orientation on $S_1$ and the opposite of the given orientation on $S_0$).

\smallskip
In 1967, Geroch \cite{Geroch1967} observed that topology change, i.e.\ the existence of a Lorentz cobordism $(M,g)$ between nondiffeomorphic closed $3$-manifolds, can only occur when $(M,g)$ admits a closed timelike curve. Finally, Tipler \cite{Tipler1977} proved in 1977 several nonexistence theorems for nontrivial Lorentz cobordisms which satisfy some energy condition. One of his results is the following (cf.\ also the remarks on p.~29 of \cite{Nardmann2007}):

\begin{Thm}[Tipler]
Let $S_0,S_1$ be closed connected $3$-manifolds, let $(M,g)$ be a connected Lorentz cobordism between $S_0$ and $S_1$ which satisfies $\Ric_g(v,v)>0$ for all lightlike vectors $v\in TM$ (i.e.\ nonzero vectors $v$ with $g(v,v)=0$). Then $S_0$ and $S_1$ are diffeomorphic, and $M$ is diffeomorphic to $S_0\times[0,1]$.
\end{Thm}

Note that every Lorentzian manifold $(M,g)$ which satisfies the dominant energy condition for some $\Lambda\in\R$ has the property that $\Ric_g(v,v)\geq0$ holds for all lightlike vectors $v\in TM$. Thus Tipler's theorem \emph{almost} rules out nontrivial dominant energy Lorentz cobordisms.

\smallskip
When we apply Theorem \ref{maindom}, we obtain a completely different result for \emph{weak} Lorentz cobordisms \cite[Corollary 9.4]{Nardmann2007}:

\begin{Thm}
Let $S_0,S_1$ be closed orientable $3$-manifolds, let $\Lambda\in\R$. Then there exists a weak Lorentz cobordism $(M,g)$ between $S_0$ and $S_1$ which satisfies the dominant energy condition with respect to $\Lambda$ and satisfies, moreover, $\Ric_g(v,v)>0$ for all lightlike vectors $v\in TM$.
\end{Thm}
\begin{proof}[Sketch of proof]
By Reinhart's theorem, there exists an orientable Lorentz cobordism $(M,g)$ between $S_0$ and $S_1$. Since every (weak) Lorentz cobordism is time-orientable by definition, we can apply Theorem \ref{maindom} (with $K=M$) and get a Lorentz metric $g'$ on $M$ which satisfies the dominant energy condition with respect to $\Lambda$ and makes every $g$-timelike vector timelike. The latter property implies that $(M,g')$ is a weak Lorentz cobordism between $S_0$ and $S_1$. The former property yields already $\Ric_{g'}(v,v)\geq0$ for all lightlike vectors $v$. A closer look at the curvature estimates in the proof of Theorem \ref{maindom} reveals that one can even arrange that $\Ric_{g'}(v,v)>0$ holds for all lightlike $v$.
\end{proof}

If $S_0,S_1$ in the previous theorem are not diffeomorphic, then Tipler's theorem shows that the weak Lorentz cobordism produced by our theorem \emph{must} make some tangent vectors to the boundary lightlike or timelike; see also Figure \ref{cobordism}. (Nevertheless, there exists a timelike vector field which is transverse to the boundary.) The flexibility required for topology change via a dominant energy metric can only be obtained from the absence of compact spacelike hypersurfaces.

\subsection{Outlook}
We conclude the discussion with some remarks and open problems.

\smallskip
Two obvious questions are whether the topological conditions on the $4$-manifold in Theorem \ref{maindom} can be removed, and whether one can always arrange that the dominant energy condition holds not only on a compact subset but globally on a noncompact manifold. We expect that the correct answer to both questions is \emph{yes}. The method of proof will be essentially the same, but the start metric $g$ and the function $f$ in the proof have to be chosen more carefully; for instance, one cannot expect that a \emph{constant} $f$ will suffice. Unfortunately a nonconstant $f$ makes the necessary estimates much more difficult, as the complicated formulas \ref{Ricformula} indicate.

\smallskip
Several points related to the nonexistence of space foliations remain to be clarified. First, it would be nice to have a quantitative criterion saying that if the spacelike region of a Lorentzian metric $g$ is on an open set $U$ in a suitable sense sufficiently close --- here we want an explicit estimate --- to a nonintegrable distribution, then $U$ does not admit a $g$-space foliation. Whereas in the previous sections, our constructions of metrics without space foliations always used a sequence $(g_k)_{k\in\N}$ of metrics and stated that there exists a $k_0$ such that for each $k\geq k_0$, the metric $g_k$ does not admit a space foliation; but we did not know explicitly how large $k_0$ had to be chosen.

\smallskip
Second, we would like to prove more than absence of space foliations, namely even absence of single spacelike hypersurfaces with certain properties. For example, consider in Minkowski space-time $\R^{3,1}$ an infinitely long cylinder $Z$ with timelike axis and, say, Euclidean radius $1$. When we perform our standard space foliation-removing procedure inside of $Z$ but keep the Minkowski metric fixed outside of $Z$, can we construct a Lorentzian metric which is equal to the Minkowski metric outside of $Z$, but does not admit any asymptotically flat spacelike hypersurface diffeomorphic to $\R^3$?

\smallskip
Lorentzian metrics without space foliations violate the usual causality assumptions of general relativity, but these causality assumptions seem to be true for the physical space-time metric of the universe we live in, within the range of currently available experimental data. Thus metrics without space foliations can only be physically realistic when the causality violations occur on length scales which are too small (or too large) to be observed. For instance, if in Theorem 3.1 $A$ is the complement of a tiny ball in Minkowski space-time, then the metric $g'$ violates e.g.\ global hyperbolicity, but since $g'$ is equal to the Minkowski metric outside of the tiny ball, this violation could hardly be detected.

\smallskip
It is widely believed in physics that at the Planck length scale ($\approx 10^{-35}$ m), our universe is no longer modelled adequately by a Lorentzian manifold. The physical relevance of the geometric considerations above depends on the speculation that there exists a length scale between $10^{-35}$ m and the experimentally accessible scale of $10^{-19}$ m, so large that space-time is already modelled to high precision by a smooth Lorentzian metric which satisfies the dominant energy condition with respect to the cosmological constant $\Lambda$, but still small enough for causality conditions to be violated. This physical speculation is not as far-fetched as it might appear at a first glance, but we cannot discuss it here. Astronomical observations suggest that $\Lambda$ is positive in the universe we live in. Note that $\Lambda>0$ is the hardest case for the construction of dominant energy metrics.

\smallskip
In order to find space-foliationless dominant energy metrics $g'$ that violate causality only on sets which are small in some sense, one would like to combine the ``relative'' Theorem \ref{foliationless} (where the metric is not changed on a given set $A$) with the existence theorem \ref{maindom} for dominant energy metrics: If $A$ is a suitable closed subset of a Lorentzian manifold $(M,g)$ such that $g$ satisfies the dominant energy condition on $A$, one wants to find a metric $g'$ on $M$ which is equal to $g$ on $A$ and satisfies the dominant energy condition everywhere.

\smallskip
A closely related question is whether topology change can occur when the space-time metric violates causality only on a very small set (too small to be detected by physicists). One would like to construct a weak Lorentz cobordism $(M,g)$ between nondiffeomorphic $3$-manifolds which satisfies the dominant energy condition and is ``almost a Lorentz cobordism'' in the sense that only a tiny set of $\mfbd M$ is not $g$-spacelike.

\smallskip
Finally, an important question is whether there exists a Lorentzian \emph{$\Lambda$-vacuum manifold} $(M,g)$ --- i.e., a Lorentzian manifold with $\Ric -\tfrac{1}{2}sg +\Lambda g=0$; in other words, a Lorentzian Einstein manifold --- which does not admit a space foliation.


\Thanks{{\em{Acknowledgments:}} We would like to thank Margarita Kraus for helpful
comments on the manuscript.}



\def\dbar{\leavevmode\hbox to 0pt{\hskip.2ex \accent"16\hss}d}
\providecommand{\bysame}{\leavevmode\hbox to3em{\hrulefill}\thinspace}
\providecommand{\MR}{\relax\ifhmode\unskip\space\fi MR }
\providecommand{\MRhref}[2]{%
  \href{http://www.ams.org/mathscinet-getitem?mr=#1}{#2}
}
\providecommand{\href}[2]{#2}

\end{document}